\documentclass[oneside,english]{amsart}
\usepackage[T1]{fontenc}
\usepackage[latin9]{inputenc}
\usepackage{color}
\usepackage{babel}
\usepackage{bbding}
\usepackage{amstext}
\usepackage{amsthm}
\usepackage{amssymb}
\usepackage[authoryear]{natbib}
\PassOptionsToPackage{normalem}{ulem}
\usepackage{ulem}
\usepackage[unicode=true,
 bookmarks=false,
 breaklinks=false,pdfborder={0 0 1},backref=section,colorlinks=false]
 {hyperref}
\hypersetup{
 colorlinks,citecolor=blue,urlcolor=blue}

\makeatletter

\providecommand{\tabularnewline}{\\}

\numberwithin{equation}{section}
\numberwithin{figure}{section}
\theoremstyle{plain}
\newtheorem{thm}{\protect\theoremname}
\theoremstyle{plain}
\newtheorem{prop}[thm]{\protect\propositionname}
\theoremstyle{plain}
\newtheorem{cor}[thm]{\protect\corollaryname}
\theoremstyle{definition}
\newtheorem{example}[thm]{\protect\examplename}
\theoremstyle{plain}
\newtheorem{lem}[thm]{\protect\lemmaname}
\theoremstyle{definition}
\newtheorem{defn}[thm]{\protect\definitionname}

\usepackage{bbm}

\DeclareMathOperator{\E}{\mathbb{E}}
\DeclareMathOperator{\R}{\mathbb{R}}
\DeclareMathOperator{\N}{\mathbb{N}}
\DeclareMathOperator{\1}{\mathbbm{1}}

\DeclareMathOperator{\CU}{\mathcal{U}}
\DeclareMathOperator{\CG}{\mathcal{G}}
\DeclareMathOperator{\CQ}{\mathcal{Q}}

\DeclareMathOperator{\val}{val}
\DeclareMathOperator{\marg}{marg}

\usepackage{babel}
\providecommand{\corollaryname}{Corollary}
\providecommand{\definitionname}{Definition}
\providecommand{\examplename}{Example}
\providecommand{\lemmaname}{Lemma}
\providecommand{\propositionname}{Proposition}
\providecommand{\theoremname}{Theorem}

\makeatother

\providecommand{\corollaryname}{Corollary}
\providecommand{\definitionname}{Definition}
\providecommand{\examplename}{Example}
\providecommand{\lemmaname}{Lemma}
\providecommand{\propositionname}{Proposition}
\providecommand{\theoremname}{Theorem}

\begin{document}
\title{Value-Based Distance Between The Information Structures}
\date{\today}
\author{Fabien Gensbittel, Marcin P\k{e}ski, Jérôme Renault}
\begin{abstract}
We define the distance between two information structures as the largest
possible difference in the value across all zero-sum games. We provide
a tractable characterization of the distance, as the minimal distance
between 2 polytopes. We use it to show various results about the relation
between games and single-agent problems, the value of additional information,
informational substitutes, complements, etc. We show that approximate
knowledge is similar to approximate common knowledge with respect
to the value-based distance. Nevertheless, contrary to the weak topology,
the value-based distance does not have a compact completion: there
exists a sequence of information structures, where players acquire
more and more information, and $\varepsilon>0$ such that any two
elements of the sequence have distance at least $\varepsilon$. This
result answers by the negative the second (and last unsolved) of the
three problems posed by J.F. Mertens in his paper ``Repeated Games",
ICM 1986. 
\end{abstract}

\thanks{F. Gensbittel and J. Renault : Toulouse School of Economics, University
Toulouse Capitole. M. P\k{e}ski : Department of Economics, University
of Toronto. J. Renault gratefully acknowledges funding from ANR-3IA Artificial and Natural Intelligence Toulouse Institute, 
grant ANR-17-EUR-0010 (Investissements d'Avenir program) and ANR MaSDOL. M. P\k{e}ski gratefully acknowledges the financial support from the SSHRC Insight Grant and the hospitality of HEC Paris,
where parts of this research were completed. }
\maketitle

\section{Introduction}

The role of information is of fundamental importance for the economic
theory. It is well known that even small differences in information
may lead to significant differences in the behavior (\citet{rubinstein_electronic_1989}).
A recent literature on the strategic (dis)-continuities has studied
this observation very intensively and in full generality. The approach
is to typically consider all possible information structures, modeled
as elements of an appropriately defined universal information structure,
and study the differences in the strategic behavior across all games.

A similar methodology has not been applied to study the relationship
between the information, and the agent's bottom line, their payoffs.
There are perhaps few reasons for this. First, following \citet{dekel_topologies_2006},
\citet{weinstein2007structure} and others, the literature has focused
on the interim rationalizability as the solution concept. Compared
with the equilibrium, this choice has several advantages: it is easier
to analyze, it is more robust from the decision-theoretic perspective,
it can be factorized through the Mertens-Zamir hierarchies of beliefs
(\citet{dekel_topologies_2006}, \citet{ElyPeski06}), and, it does
not suffer from the existence problems (unlike the equilibrium - see
\citet{simon_games_2003}). However, the value of information is typically
measured in the ex ante sense, where solution concepts like the Bayesian
Nash equilibrium being more appropriate. Second, the multiplicity
of solutions necessitates that the literature takes the set-based
approach. This, of course, makes the quantitative comparison of the
value of information difficult. Last but not least, the freedom in
choosing games without any restriction makes the payoff comparison
between information structures useless (see Section \ref{subsec:Payoff distance NZS}
for a detailed discussion of this point).

Nevertheless, we find the questions about the value of a general-formulated
information as important and fascinating. How to measure the value
of information on the universal type space? How much a player can
gain (or lose) from an additional information? Which information structures
are similar, in the sense that they always lead to the same payoffs?
In order to address these questions, we propose to restrict the analysis
to zero-sum games. It is a natural restriction for both conceptual
and methodological reasons. The question of the value of information
is of special importance when the players' interests are opposing.
With zero-sum games, the information has natural comparative statics:
a player is better off when her information improves and/or the opponent's
information worsens (\citet{peski_comparison_2008}). Such comparative
statics are intuitive, they hold in the single-agent decision problems
(\citet{blackwell1953}), but they do not hold for general games,
where better information may worsen a player's strategic position.
Moreover, the restriction avoids some of the problems mentioned in
the previous paragraph. Finite zero-sum games have always an equilibrium
on common prior information structures (\citet{mertens_sorin_zamir_2015})
that depends only on the distribution over hierarchies of beliefs.
The equilibrium has decent decision-theoretic foundations (\citet{brandt2019justifying}).
Even if the equilibrium is not unique, the ex ante payoff always is
and it is equal to the value of the zero-sum game. Finally, the restriction
leads to novel issues comparing to the earlier literature, where many
of the results rely on either coordination games, or betting games
(\citet{rubinstein_electronic_1989}, \citet{morris_typical_2002},
\citet{ely_critical_2011}, \citet{chen_e-mail_2013} among others).We
believe that it is important to reconsider various phenomena to clarify
their relevance for different classes of games.

We define the distance between two information structures as the largest
possible difference in the value across all zero-sum payoff functions
that are bounded by a constant. This has a straightforward interpretation
as a tight upper bound on the gain or loss from moving from one information
structure to another. The first result provides a characterization
of the distance as the total variation distance between two sets of
information structures: those that are better for player one than
one of the original information structures and those that are worse
than the other one. This distance can be computed as a solution to
a finitely dimensional convex optimization problem.

The characterization is tractable in applications. In particular,
we use it to describe the conditions under which the distance between
information structures is maximized in single-agent problems (which
are special class of zero-sum games). We recover the characterization
of a comparison of information structures from \citet{peski_comparison_2008},
which generalized to 2 players the Blackwell's comparison of experiments
via garblings. We also discuss the role of the marginal distribution
over the state.

An important aspect of our approach is its natural and simple interpretation.
It allows us to define the value of an additional piece of information
as the distance between two type spaces, in one of which one or two
players have an access to new information. The above characterization
of the distance allows us to prove numerous results about the value.
We give conditions when the value of new information is maximized
in the single-agent problems. Next, we describe the situations when
the value of one piece of information decreases when the other piece
of information becomes available, in other words, when the two pieces
of information are substitutes. Similarly, we show that, under some
conditions, the value of one piece of information increases when the
other player gets an additional information, or in other words, that
the two pieces of information, one for each player, are complements.
Finally, we show that the new information matters only if it is valuable
to at least one of the players individually. The joint information
contained in the correlation between player's signals is in itself
not valuable in the zero-sum games.

Apart from its quantitative metric aspect, our distance contains interesting
topological information. In a striking example, we show that any information
structure in which with a large (close to 1) probability, each player
assigns a large (close to 1) probability to some state is similar
to a structure in which the state is publicly revealed. In other words,
the distinction between approximate knowledge and approximate common
knowledge (\citet{mondsam:89}) is not relevant for the value-based
distance. There is a simple intuition for this. The common knowledge
is a statement about the joint information. In order to benefit from
it, the players' interests need to be somehow aligned, like in coordination
games. That is of course impossible in a situation of conflict.

More generally, we show that any sequence of countable information
structures converges to a countable structure if and only if the associated
hierarchies of beliefs converge in the induced weak topology of \citet{mertzam:85}.
Thus, at least in the neighborhoods of countable structures, the higher
order beliefs matter less and less in zero-sum games. This leads to
a question whether the higher order beliefs matter at all. Perhaps
surprisingly in the light of the above results, the answer is affirmative.
We demonstrate it by constructing an infinite sequence of information
structures $u^{n}$, such that all the information structures $n^{\prime}\geq n$
have the same $n$-the order hierarchies of beliefs. We show that
there exists $\varepsilon>0$ such that the value-based distance between
each pair of structures is at least $\varepsilon$. In the proof,
we construct a Markov chain with the first element of the chain is
correlated with the state of the world. We construct an information
structure $\mu^{n}$ so that player 1 observes the first $n$ odd
elements of the sequence and the other player observes the first $n$
even elements. Our construction implies that in information structure
$\mu^{n+1}$ each player gets an extra signal. Thus, having more and
more information may lead... nowhere. This is unlike the single-player
case, where more and more signals is a martingale and the value converges
for each decision problem. We conclude that our distance is not robust
with respect to the correct specification of higher order beliefs.

The last result sounds similar to the results from the strategic (dis)continuities
literature. However, we emphasize that our result is entirely novel.
In particular, and it has to deal with at least two major difficulties.
The first difficulty is that all earlier constructions heavily relied
on non-zero sum games: with either coordination games, or betting
elements. Such constructions do not work in zero-sum games. Another
difficulty is that when we construct games in which the players' payoffs
depend on their higher-order beliefs, we are constrained by an uniform
payoff bound.

An important contribution of our result is that it leads to an answer
to the last open problem posed in \citet{mertens_repeated_1986}\footnote{Problem 1 asked about the convergence of the value, and it was proved
false in \citet{ziliotto_zero-sum_2016}. Problem 3 asked about the
equivalence between the existence of the uniform value and the uniform
convergence of the value functions, it was proved to be false by \citet{monderer_asymptotic_1993}
and \citet{lehrer_discounting_1994}.}. Specifically, his Problem 2 asks about the equicontinuity of the
family of value functions over information structures across all (uniformly
bounded) zero-sum game. The positive answer would have imply the equicontinuity
of the discounted and the average value in repeated games, and it
would have consequences for the convergence in the limits theorems\footnote{\label{foot_zerosum} Equicontituity of value functions is used to
obtain limit theorems in several works as e.g. \citet{mertens1971value},
\citet{forges1982infinitely}, %\citet{sorin1984big,sorin1985big},
\citet{rosenberg2001operator}, \citet{rosenberg2000zero}, \citet{rosenberg2000maxmin},
\citet{rosenberg2004stochastic}, \citet{renault_value_2006},
\citet{gensbittel_value_2012}, \citet{venel2014commutative}, \citet{renault_venel_2016}.
See section \ref{subsec:Mertens} for a more detailed discussion}. One can show that the problem 2 is equivalent to a question whether
the value-based distance is totally bounded\footnote{Recall that a metric space if for all $\varepsilon>0$, it can be
covered by finitely many balls of radius $\varepsilon$.} on countable information structures. Unfortunately, our results imply
that the answer to the problem is negative. In particular, it is not
possible to approximate the universal information structure with finitely
many well-chosen information structures.

Our paper adds to the literature on the topologies of information
structures. This literature was spurred by an observation in \citet{rubinstein_electronic_1989}
that solution concepts are highly sensitive to higher-order beliefs.
\citet{dekel_topologies_2006} introduce \emph{strategic} and \emph{uniform-strategic
topologies. }In the latter, two types are close if, for any (not necessarily
zero-sum) game, the sets of (almost) rationalizable outcomes are (almost)
equal (see also \citet{morris_typical_2002}). \footnote{ \citet{dekel_topologies_2006}
focus mostly on the weaker notion of \emph{strategic topology} that
differs from the uniform strategic in the same way that the pointwise
convergence differs from uniform convergence.} \citet{chen_uniform_2010} and \citet{chen_uniform_2016} provide
a characterization of the uniform-strategic topology in terms of the
uniform weak topology on belief hierarchies. There are two key differences
between that and our approach. First, the uniform strategic topology
applies to all (including non-zero sum) games. Our restriction allows
us to clarify which of the results established in the literature hold
in situations of pure conflict. Second, we work with \emph{ex ante}
information structures and the equilibrium solution concept, whereas
the uniform strategic topology is designed to work on the \emph{interim}
level, with rationalizability. The ex ante equilibirum approach is
more appropriate for the value comparison and other related questions.
For instance, in the information design context, the quality of he
information structure is typically evaluated \emph{before} players
receive any information.

The value of information literature studies the impact of information
on the payoffs in various classes of games. Examples include single
agent problems (\citet{blackwell1953}, \citet{athey2018value}), zero-sum games
(\citet{gossner_value_2001}, \citet{shmaya_value_2006}), common
interest games \citet{lehrer2010signaling,lehrer2013garbling}, non-zero-sum static games (\citet{gossner2000comparison}),
Markov games (\citet{renault_value_2006}, \citet{peski_toikka}),
among many others. Our paper contributes with the characterization
of the tight upper bound on the loss/gain from moving from one information
structure to another. The characterization allows us to discuss various
results about the value of additional information. In particular,
we characterize the situations when the two pieces of information
are substitutes or complements in zero-sum games (\citet{hellwig2009knowing}
study when the information acquisition decisions are complements or
substitutes in a beauty contest game).

Finally, this paper contributes to a recent but rapidly growing field
of information design (\citet{kamenica_bayesian_2011}, \citet{ely_beeps_2015},
\citet{bergemann_bayes_2015}, to name a few). In that literature,
a principal designs an information structure which the agents use
to play a game with incomplete information. The objective is to maximize
the principal's payoff from the equilibrium outcome of some game.
Sometimes, the design of information may be divorced from the game
itself. For example, a bank may acquire a software to process and
analyze large amounts of financial information before knowing what
stock they are going to trade on, or, a spy master allocates resources
to different tasks or regions before she understands the nature of
future conflicts. Our result shows that the choice space of information
structures is large: there is no natural way in approximating the
space of choices with a finite set of structures.

Section \ref{sec:Model} defines the value-based distance. Section
\ref{sec:Computing} provides and discusses the characterization of
the distance as a total variation distance between two sets of measures.
Section \ref{sec:Value-of-additional} lists various results about
the value of additional information. We discuss the topological content
of the value-based distance in Section \ref{sec:Topology}. Section
\ref{sec:Compactness} shows that the space of countable information
structures is not totally bounded for the value-based distance. Section
\ref{sec:Relation-to-other} discusses the relation to other topologies
on the space of hierarchies of beliefs. The last section concludes.
The proofs are contained in the Appendix.

\section{Model\label{sec:Model}}

Let $K$ be a finite set, $\left|K\right|\geq2$. A (countable) \emph{information
structure} is an element $u\in\Delta\left(K\times\N\times\N\right)$
of the space of probabilities over tuples $\left(k,c,d\right)\in K\times\N\times\N$,
with the interpretation that $k$ is a state of the world, and $c$
and $d$ are the signals of, respectively, player 1 and player 2.
%that come from at most countable sets $C,D$. 
In other words, an information structure is a 2-player common prior
Harsanyi type space over $K$ with at most countably many types. The
set of information structures is denoted by $\CU=\Delta(K\times\N\times\N)$.

We will identify any probability over $K\times C\times D$ where $C,D$
are at most countable sets with an element of $\CU$, where we interpret
$C$ and $D$ as subsets of $\N$\footnote{ More precisely, we associate with every set $C$ which is at most
countable an enumeration, i.e. a bijective map $\phi_{C}$ between
$C$ and $\{0,...,|C|-1|\}$ when $C$ is finite or $\N$ when $C$
is infinite and identify the information structure $u\in\Delta\left(K\times C\times D\right)$
with the distribution of $(k,\phi_{C}(c),\phi_{D}(d))$ induced by
$u$. All our results are independent of the choice of these enumerations } . For $L=1,2,...$, let $\CU\left(L\right)$ be the subset of information
structures where each player receives a signal smaller or equal to
$L-1$ with probability $1$, so that each player has at most $L$
different signals.

Whereas previous papers in the literature restrict attention\footnote{For instance, one can read in \citet{BDM2010} ``We leave open the
question of what happens when the components of the state on which
the players have some information fail to be independent.... In this
situation the notion of monotonicity is unclear, and the duality method
is not well understood."} to a particular subset of $\CU$ (independent information, lack of
information on one side, fixed support...), we will study the general
case of information structures in $\CU$.

We evaluate information structures via the values of associated zero-sum
Bayesian games.\emph{ }A\emph{ payoff function} is a map $g:K\times I\times J\to[-1,1]$,
where $I,J$ are finite sets of actions. The set of payoff functions
with action sizes of cardinality $\leq L$ is denoted by $\CG(L)$,
and let $\CG=\bigcup_{L\geq1}\CG(L)$ be the set of all payoff functions.
A information structure $u$ and a payoff function $g$ together define
a zero-sum Bayesian game $\Gamma(u,g)$ played as follows: First,
$(k,c,d)$ is selected according to $u$, player 1 learns $c$ and
player 2 learns $d$. Then simultaneously player 1 chooses $i\in I$
and player 2 chooses $j\in J$, and finally the payoff of player 1
is $g(k,i,j)$. The zero-sum game $\Gamma(u,g)$ has the value (the
unique equilibrium, or the minmax payoff of player 1) which we denote
by $\val(u,g)$. We sometimes refer to player 1 as the maximizer,
and to her opponent as the minimizer.

We define \emph{the value-based distance} between two information
structures as the largest possible difference in the value across
all payoff functions: 
\begin{equation}
d(u,v)=\sup_{g\in\CG}|\val(u,g)-\val(v,g)|.\label{eq:ZS distance}
\end{equation}

The distance (\ref{eq:ZS distance}) satisfies two axioms of a metric:
the symmetry, and the triangular inequality. However, it is possible
that $d(u,v)=0$ for $u\neq v$. For instance, if we start from an
information structure $u$ and relabel the signals of the players,
we obtain an information structure $u'$ which is formally different
from $u$, but ``equivalent'' to $u$. Say that $u$ and $v$ are
equivalent, and write $u\sim v$, if for all game structures $g$
in $\CG$, $\val(u,g)=\val(v,g)$. We let $\CU^{*}=\CU/\sim$ be the
set of equivalence classes. Thus, $d$ is a pseudo-metric on $\CU$
and a metric on $\CU^{*}$.

In order to state our main results, we will also use the total variation
norm given for each $u,v\in\CU$ by: 
\[
\|u-v\|=\sum_{k,c,d}|u(k,c,d)-v(k,c,d)|.
\]
Since all payoffs are in $[-1,1]$ it is easy to see that $d(u,v)\leq\|u-v\|\leq2$.
It is also convenient to always identify $I$ and $J$ with the sets
$\{0,...,|I|-1\}$ and $\{0,...,|J|-1\}$ using some enumeration and
to extend $g$ on $K\times\N\times\N$ by letting $g(k,i,j)=-1$ if
$i\geq|I|$ and $g(k,i,j)=1$ if $i<|I|$ and $j\geq|J|$. Note that
the value of the game $\Gamma(u,g)$ is not modified and that optimal
strategies remain optimal if we allow the players to take any action
in $\N$. With this convention, each set $\CG(L)$ and thus also $\CG$
is identified with a subset of the set of maps from $K\times\N\times\N$
to $[-1,1]$.

\section{Computing $d(u,v)$\label{sec:Computing}}

We give here a tractable characterization of $d(u,v)$ and we illustrate
it with some applications.

\subsection{Characterization of $d\left(u,v\right)$}

We start with the notion of garbling, used by Blackwell to compare
statistical experiments \citet{blackwell1953}. A \emph{garbling}
is an element $q:\N\to\Delta(\N)$, and the set of all garblings is
denoted by $\CQ$. Given a garbling $g$ and an information structure
$u$, we define the information structures $q.u$ and $u.q$ so that
for each $k,c,d$, 
\begin{align*}
q.u(k,c,d) & =\sum_{c^{\prime}}u(k,c{}^{\prime},d)q(c{}^{\prime}|c),\text{ and}\\
u.q(k,c,d) & =\sum_{d^{\prime}}u(k,c,d{}^{\prime})q(d'|d).
\end{align*}
We also denote by $\CQ(L)$ the subset of garblings $q:\N\to\Delta(\{0,...,L-1\})$.
There are two interpretations of a garbling. First, the garbling can
be seen as an information loss: suppose that $(k,c,d)$ is selected
according to $u$, $c'$ is selected according to the probability
$q(c)$, and player 1 learns $c'$ (and player 2 learns $d$). The
new information structure is exactly equal to $q.u$, where the signal
received by player 1 has been deteriorated through the garbling $q$.
Similarly, $u.q$ corresponds to the dual situation where the signal
of player 2 has been deteriorated. Further, the garbling $q$ can
also be seen as a behavior strategy of a player in a Bayesian game
$\Gamma(u,g)$: if the signal received is $c$, play the mixed action
$q(c)$. The relation between the two interpretations plays an important
role in the proof of Theorem \ref{thm1} below. 
\begin{thm}
\label{thm1} For each $u,v\in\CU$ for 
\begin{eqnarray}
\sup_{g\in\CG}\left(\val(v,g)-\val(u,g)\right) & = & \min_{q_{1}\in\CQ,q_{2}\in\CQ}\|q_{1}.u-v.q_{2}\|.\label{eq:d characterization}
\end{eqnarray}
If $u,v\in\CU\left(L\right)$ for $L\in\N$, then all the optima can
be attained by $L$-based structures (games and garblings). Hence,
\[
d(u,v)=\max\left\{ \min_{q_{1}\in\CQ,q_{2}\in\CQ}\|q_{1}.u-v.q_{2}\|,\min_{q_{1}\in\CQ,q_{2}\in\CQ}\|u.q_{1}-q_{2}.v\|\right\} .
\]
\end{thm}

The Theorem provides a characterization of the value-based distance
between two information structures $u$ and $v$ for each player as
a total variation distance between two sets obtained as garblings
of the original information structures $\left\{ q.u:q\in\CQ\right\} $
and $\left\{ v.q:q\in\CQ\right\} $. If the two information structures
are in $\CU(L)$, then all the suprema in the first line can be achieved
by payoffs functions in $\CG(L)$, and garblings in $\CQ(L)$.

The Theorem simplifies the problem of computing the value-based distance
in two ways. First, it reduces the dimensionality of the optimization
domain from payoff functions and strategy profiles (to compute value)
to a pair of garblings. More importantly, notice that the solution
to the original problem (\ref{eq:ZS distance}) is typically a saddle-point
as it involves finding optimal strategies in a zero-sum game. On the
other hand, the function $\left\Vert q_{1}.u-v.q_{2}\right\Vert $
is convex in garblings $(q_{1},q_{2})$, and the domains $\CQ$ are
convex compact sets. Thus, the right-hand side of (\ref{eq:d characterization})
is a convex optimization problem. Moreover, if $u,v\in\CU(L)$ for
some $L\in\N$, this convex problem is equivalent to a finite dimensional
polyhedral convex problem. The characterization of the distance is
quite tractable. We show it in numerous applications in this and subsequent
sections.

\subsection{Comments and applications}

\subsubsection{Single-agent problems}

A special case of a zero-sum game is a single-agent problem. Formally,
a payoff function $g\in\CG(L)$ is a \emph{single-agent (player 1)
problem} if $g\left(k,i,j\right)=g\left(k,i,j^{\prime}\right)$ for
any $k\in K$ and any $i,j,j^{\prime}\in\{0,...,L-1\}$ (or alternatively,
if the set of actions of player 2 is a singleton, $J=\left\{ *\right\} $).
Let $\CG_{1}\subset\CG$ be the set of player 1 problems. Then, for
each $g\in\CG_{1}$, each information structure $u$, $\val\left(g,u\right)$
is the maximal expected payoff of player 1 in problem $g$. Let 
\begin{equation}
d_{1}\left(u,v\right):=\sup_{g\in\CG_{1}}\left|\val\left(u,g\right)-\val\left(v,g\right)\right|\leq d\left(u,v\right).\label{eq:single agent distance}
\end{equation}
Analogously, we can define $d_{2}\left(u,v\right)$ as the distance
measured using only player 2 single agent problems.

A simple consequence of the definition is that the single-agent distance
depends only on the distribution of the posterior beliefs of player
$1$ about the state. For each $u\in\CU$, let $\tilde{u}\in\Delta(\Delta(K))$
denote the distribution of the posterior beliefs of player $1$ about
the state induced by $u$, and $D$ be the set of suprema of affine
functions from $\Delta(K)$ to $[-1,1]$. Then, one shows that 
\begin{equation}
d_{1}(u,v)=\sup_{f\in D}\left|\int_{p\in\Delta(K)}f(p)d{\tilde{u}}(p)-\int_{p\in\Delta(K)}f(p)d{\tilde{v}}(p)\right|.\label{eq:d1weak}
\end{equation}

A natural question is when the two distances are equal, or, alternatively,
when the maximum in (\ref{eq:ZS distance}) is attained by the single-agent
problems. The next result provides a partial answer to this question.
For any structure $u\in\Delta\left(K\times C\times D\right)$, we
say that the players information is \emph{conditionally independent},
if under $u$, signals $c$ and $d$ are conditionally independent
given $k$. 
\begin{prop}
\label{prop: SA=00003D00003DZS}Suppose that $u,v\in\Delta\left(K\times C\times D\right)$
are two information structures with conditionally independent information
such that $\marg_{K\times D}u=\marg_{K\times D}v$. Then, 
\[
d\left(u,v\right)=d_{1}\left(u,v\right).
\]
\end{prop}

Proposition \ref{prop: SA=00003D00003DZS} says that, given two conditionally
independent structures, if the information of player 2 is the smae,
the distance is equal to value of the difference between the information
of player 1. The proof of the Proposition relies on the characterization
from Theorem \ref{thm1} and it shows that the minimum in the optimization
problem is attained by the same pair of garblings as in the single-agent
version of the problem.

\subsubsection{Comparison of information structures\label{subsec:Comparison-of-information}}

We partially order the information structures by how good they are
for player 1: given $u$, $v$ in $\CU$, say that \emph{player 1
prefers $u$ to $v$}, write $u\succeq v$, if for all $g$ in $\CG$,
$\val(u,g)\geq\val(v,g)$. The order $\preceq$ is reflexive and transitive
(and it is also antisymmetric on $\CU^{*}$). One shows that the value
is monotonic in the information of the players: for any garbling $q$,
player 1 always weakly prefers $u$ to $q.u$, and $u.q$ to $u$.

Theorem \ref{thm1} implies that 
\[
u\succeq v\Longleftrightarrow\exists_{q_{1},q_{2}\in\CQ}q_{1}.u=v.q_{2}.
\]
(To see why, notice that $u\succeq v$ is equivalent to the left-hand
side of the first equality in (\ref{eq:d characterization}) must
be weakly negative, which is equivalent to the second claim above.)
This characterization was initially obtained in \citet{peski_comparison_2008}
who generalized the Blackwell characterization of more informative
experiment to the multi-player setting.

Additionally, we obtain the following Corollary to Theorem \ref{thm1}: 
\begin{cor}
\label{cor: d charact comparison}For all information structures $u,v$,
\begin{equation}
\sup_{g\in\CG}\left(\val(v,g)-\val(u,g)\right)=\inf_{u'\preceq u,v'\succeq v}\|u'-v'\|.\label{eq:Thm1 as comparison}
\end{equation}
\end{cor}

This observation provides an additional interpretation to the characterization
from Theorem \ref{thm1}: the maximum gain from replacing information
structure $u$ by $v$ is equal to the minimum total variation distance
between the set of information structures that are worse than $u$
and those that are better than $v$.

\subsubsection{The impact of the marginal over $K$}

Among many ways that two information structures can differ, the most
obvious one is that they may have different distributions over the
states $k$. In order to capture the impact of such differences, the
next result provides tight bounds on the distance between two type
spaces with a fixed distribution of a state: 
\begin{prop}
\label{prop:diameter} For each $p,q\in\Delta K$, 
\begin{align}
\min_{u,v:\marg_{K}u=p,\marg_{K}v=q}d\left(u,v\right) & =\sum_{k}\left|p_{k}-q_{k}\right|,\label{eq:diameter bounds}\\
\max_{u,v:\marg_{K}u=p,\marg_{K}v=q}d\left(u,v\right) & =2\left(1-\max_{p^{\prime},q{}^{\prime}\in\Delta K}\sum_{k}\min\left(p_{k}q{}_{k}^{\prime},p{}_{k}^{\prime}q_{k}\right)\right).\nonumber 
\end{align}
If $p=q$, the upper bound is equal to 
\begin{align*}
\max_{u,v:\marg_{K}u=\marg_{K}v=p}d(u,v) & =2(1-\max_{k}p_{k}).
\end{align*}
\end{prop}

The lower bound in (\ref{eq:diameter bounds}) is reached when the
two information structures do not provide any information to none
of the players. The upper bound is reached with information structures
where one player knows the state perfectly, and the other player does
not know anything. 
\begin{example}
\label{ex: extreme distance}Suppose that $K=\left\{ 0,1\right\} $.
Fix $p,q\in\Delta K$. In such a case, one easily checks that the
maximum in the right hand side is attained by either $p^{\prime}=q{}^{\prime}=(1,0)$,
or $p^{\prime}=q{}^{\prime}=(0,1)$, or $p^{\prime}=p,q{}^{\prime}=q$.
It follows that for any two information structures $u,v$, 
\[
d\left(u,v\right)\leq2\left(1-\max\left(\min\left(p_{0},q_{0}\right),\min\left(p_{1},q_{1}\right),p_{0}q_{0}+p_{1}q_{1}\right)\right).
\]
The bound is attained when, for example, $u=\bar{u},v=\underline{v}$
and $\bar{u}\left(k,c,d\right)=p_{k}\1_{c=k}\1_{d=0}$ and $\underline{v}\left(k,c,d\right)=q_{k}\1_{c=0}\1_{d=k}$
for each $k,c,d\in\left\{ 0,1\right\} $.
\end{example}

\subsubsection{Optimal strategies}

Another useful property of the characterization is that the garblings
that solve (\ref{eq:d characterization}) can be used to transform
optimal strategies in one structure to approximately optimal strategies
on another structure. More precisely, we say that strategy $\sigma$
of player 1 is $\varepsilon$-optimal in game $g$ on structure $u$
if for any strategy $\tau$ of player 2, the payoff of $\sigma$ against
$\tau$ is no smaller than $\val\left(u,g\right)-\varepsilon$. We
similarly define $\varepsilon$-optimal strategies for player 2.

For a strategy $\sigma\in\CQ$ and a garbling $q_{1}\in\CQ$, define
$\sigma.q_{1}$ in $\CQ$ by $\sigma.q_{1}(c)=\sum_{c^{\prime}}q_{1}(c{}^{\prime}|c)\sigma(c{}^{\prime})$
for each signal $c$: player 1 receives signal $c$, then selects
$\chi$ according to $q_{1}(c)$ and plays $\sigma(\chi)$. We have 
\begin{prop}
\label{prop: optimal strategies}Fix finite $u,v$ and suppose that
$q_{1}$ and $q_{2}$ are solutions to (\ref{eq:d characterization}).
Then, if $\sigma$ is an optimal strategy in $g$ on $v$, then $\sigma.q_{1}$
is a $2d\left(u,v\right)$-optimal strategy in $g$ on $u$. Similarly,
if $\tau$ is optimal for player 2 in $g$ on $u$, then $\tau.q_{2}$
is $2d\left(u,v\right)$-optimal for player 2 in $g$ on $v$. 
\end{prop}

\section{Value of additional information\label{sec:Value-of-additional}}

An important advantage of our approach is that the distance $d\left(u,v\right)$
has a straightforward interpretation as the maximum willingness to
pay to move from one information structure to another. This allows
us to discuss the value of additional information, which we define
as the the distance between structures with or without an extra piece
of information.

\subsection{Value of additional information: games vs. single agent }

Consider two information structures 
\begin{align*}
u & \in\Delta\left(K\times(C\times C{}^{\prime})\times D\right),\\
v & =\marg_{K\times C\times D}u.
\end{align*}
The structure $u$ is the same as $v$, but where player 1 can observe
an additional signal $c^{\prime}$. Because $u$ represents more information,
$u$ is (weakly) more valuable, and the value of the additional information
is defined as $d\left(u,v\right),$ i.e., the tight upper bound on
the gain from the additional signal. A corollary to Proposition \ref{prop: SA=00003D00003DZS}
shows that if the signals of the two players ate independent conditionally
on the state, the gain from the new information is the largest in
the single-agent problems. 
\begin{cor}
\label{prop: Games vs Single agent}Suppose that information in $u$
(and therefore in $v$) is conditionally independent. Then, 
\[
d\left(u,v\right)=d_{1}\left(u,v\right).
\]
\end{cor}

The next example illustrates that the thesis of Proposition \ref{prop: Games vs Single agent}
does not hold without a conditional independence. 
\begin{example}
\label{ex value added}Suppose that $C=\left\{ *\right\} ,K=D=C{}^{\prime}=\left\{ 0,1\right\} $,
and 
\[
u\left(k,d,c^{\prime}\right)=\begin{cases}
\frac{1}{4}\frac{k+c^{\prime}}{2}, & \text{if }d=1,\\
\frac{1}{4}\left(1-\frac{k+c^{\prime}}{2}\right), & \text{if }d=0.
\end{cases}
\]
Then, the new information $c^{\prime}$ is independent from $k$,
but $d$ is both correlated with $k$ and $c^{\prime}$. Because $c^{\prime}$
is independent from the state, the value of new information in player
1 single-agent problem is 0: 
\[
d_{1}\left(u,v\right)=0.
\]
However, signal $c^{\prime}$ provides non-trivial information about
the signal of the other player, hence it is valuable in some games.
It is easily seen from \ref{thm1} that $d(u,v)>0$. Indeed, since
$u$ gives player $1$ more information, we have 
\[
d(u,v)=\min_{q_{1},q_{2}}\|u.q_{2}-q_{1}.v\|,
\]
where $q_{2}:\{0,1\}\rightarrow\Delta\{0,1\}$ and $q_{1}\in\Delta\{0,1\}$.
The existence of a pair $(q_{1},q_{2})$ such that $\|u.q_{2}-q_{1}.v\|=0$
is equivalent to the system of equations 
\[
\forall(k,d,c')\in\{0,1\}^{3},\;u.q_{2}(k,d,c')=v.q_{1}(k,d,c'),
\]
where the unknowns are $q_{1},q_{2}$, and one can check that this
system does not admit any solution. In other words, the information
that would be useless in a single-agent decision problem is valuable
in a strategic setting. 
\end{example}

A special case of the conditionally independent information is where
players receive multiple samples of independent Blackwell experiments.
In an online Appendix, we show how to compute the value of additional
Blackwell experiments. 

\subsection{Informational substitutes}

Next, we ask two questions about the impact of a piece of information
on the value of another piece of information. Suppose that 
\begin{align*}
u & \in\Delta\left(K\times(C\times C_{1}\times C_{2})\times D\right),\\
v & =\marg_{K\times(C\times C_{1})\times D}u,\\
u^{\prime} & =\marg_{K\times(C\times C_{2})\times D}u,\\
v^{\prime} & =\marg_{K\times C\times D}u.
\end{align*}
When moving from $v^{\prime}$ to $u^{\prime}$ or $v$ to $u$, player
1 gains an additional signal $c_{2}$. The difference is that in the
latter case, player 1 has more information that comes from signal
$c_{1}$. The next result shows the impact of an additional signal
on the value of information. 
\begin{prop}
\label{prop. Info substitutes}Suppose that, under $u$, $c_{1}$
is conditionally independent from $\left(c,c_{2},d\right)$ given
$k$. Then, 
\[
d\left(u^{\prime},v^{\prime}\right)\geq d\left(u,v\right).
\]
\end{prop}

Given the assumptions, the marginal value of signal $c_{2}$ decreases
when signal $c_{1}$ is also present. In other words, the two pieces
of information are substitutes.

The conditional assumption is equivalent to two simpler assumptions
(a) $c_{1}$ is conditionally independent from $\left(c,c_{2}\right)$
given $\left(k,d\right)$, and (b) $c_{1}$ and $d$ are conditionally
independent given $k$. Both (a) and (b) are important as it is illustrated
in the two subsequent examples. 
\begin{example}
Violation of (a). Suppose that $C=D=\left\{ *\right\} $, $K=C_{1}=C_{2}=\left\{ 0,1\right\} $,
$c_{1}$ and $c_{2}$ are uniformly and independently distributed,
and $k=c_{1}+c_{2}\mod2$. Then, signal $c_{2}$ is itself independent
from the state, hence useless without $c_{1}$. Knowing $c_{1}$ and
$c_{2}$ means knowing the state, which is, of course, very valuable.
Thus, the value of $c_{2}$ increases when $c_{1}$ is also present. 
\end{example}

\medskip{}

\begin{example}
Violation of (b). Suppose that $C=\left\{ *\right\} $, $K=C_{1}=C_{2}=D=\left\{ 0,1\right\} $,
$c_{1}$ and $d$ are uniformly and independently distributed, $c_{2}=d$,
and $k=c_{1}+d\mod2$. Notice that part (a) of the assumption holds
(given $\left(k,d\right)$, both signals $c_{1}$ and $c_{2}$ are
constant, hence, independent), but part (b) is violated. Again, signal
$c_{2}$ is useless alone, but together with $c_{1}$ it allows the
determine the state and the information of the other player. 
\end{example}

\subsection{Informational complements}

Another question is about the impact of an information of the other
player on the value of information. Suppose that 
\begin{align*}
u & \in\Delta\left(K\times(C\times C_{1})\times(D\times D_{1})\right),\\
v & =\marg_{K\times C\times(D\times D_{1})}u,\\
u^{\prime} & =\marg_{K\times(C\times C_{1})\times D}u,\\
v^{\prime} & =\marg_{K\times C\times D}u.
\end{align*}
When moving from $v^{\prime}$ to $u^{\prime}$ or $v$ to $u$, in
both cases, player 1 gains an additional signal $c_{1}$. However,
in the latter case, player 2 has an additional information that comes
from signal $d_{1}$. The next result shows the impact of the opponent's
signal on the value of information. 
\begin{prop}
\label{prop. Info complements} Suppose that $\left(c,c_{1}\right)$
and $d$ are conditionally independent given $k$. Then, 
\[
d\left(u^{\prime},v^{\prime}\right)\leq d\left(u,v\right).
\]
\end{prop}

Given the assumptions, signal $c_{1}$ becomes more valuable when
the opponent has also an access to additional information. Hence,
the two pieces of information are complements. 
\begin{example}
The independence assumptions are important. Suppose that $C=\left\{ *\right\} $
and $K=D=D_{1}=C_{1}=\left\{ 0,1\right\} $. The state is drawn uniformly.
Signal $d$ is equal to the state with probability $\frac{2}{3}$
and signal $d_{1}$ is equal to the state for sure. Finally, $c_{1}=1$
iff $d=k$. In other words, player 2's signal $d$ is an imperfect
information about the state. Signal $c_{1}$ carries information about
the quality of the signal of player 1. Such signal is valuable in
some games, hence $d\left(u^{\prime},v{}^{\prime}\right)>0.$ In the
same time, if player 2 learns the state perfectly, signal $c$ becomes
useless, and $d\left(u,v\right)=0$. (The last claim can be formally
proven using Proposition \ref{prop: joint info} below.) 
\end{example}

\subsection{Value of joint information}

Finally, we consider a situation where two players receive additional
information simultaneously. We ask when the joint information contained
in signals of two players is (not) valuable. We need the following
definition. Consider a distribution $\mu\in\Delta\left(X\times Y\times Z\right)$
over countable spaces. We say that random variables $x$ and $y$
are $\varepsilon$-conditionally independent given $z$ if 
\[
\sum_{z}\mu\left(z\right)\sum_{x,y}\left|\mu\left(x,y|z\right)-\mu\left(x|z\right)\mu\left(y|z\right)\right|\leq\varepsilon.
\]

Let 
\begin{align*}
u & \in\Delta\left(K\times(C\times C_{1})\times(D\times D_{1})\right),\\
v & =\text{marg}_{K\times C\times D}u.
\end{align*}

\begin{prop}
\label{prop: joint info}Suppose that $d_{1}$ is $\varepsilon$-conditionally
independent from $k\times c$ given $d$, and $c_{1}$ is $\varepsilon$-conditionally
independent from $\left(k,d\right)$ given $c$. Then, 
\[
d\left(u,v\right)\leq\varepsilon.
\]
\end{prop}

Suppose that the new signal of each player is not providing any significant
information about the state of the world and original information
of the other player. Such a signal would be useless in a single-decision
problem. In principle, the new signals could be useful in a strategic
setting as the new signals may be correlated with each other, and
jointly correlated with the old information. (In fact, it is possible
that the new signals, if learned by both players, would completely
reveal the original information. See the example below.) Nevertheless,
the Proposition says that the information that is jointly shared by
the two players is not valuable in the zero-sum games.
\begin{example}
\label{exa6} Let $K=\left\{ 0,1\right\} $. Consider a sequence of
information structures $u_{n}\in\Delta\left(K\times\left\{ 0,...,n\right\} \times\left\{ 0,...n+1\right\} \right)$
and structure $u\in\Delta\left(K\times\left\{ 0\right\} \times\left\{ 0\right\} \right)$
such that 
\begin{align*}
u_{n}\left(0,l,l\right) & =u_{n}\left(1,l,l+1\right)=\frac{1}{2\left(n+1\right)}\text{ for each }l=0,...,n,\\
u\left(k,0,0\right) & =\frac{1}{2}\text{ for each }k=0,1.
\end{align*}
This is illustrated on the Figure below (all lines have equal probability;
the dashed line corresponds to state 0 and the solid line corresponds
to state 1). \vspace{0.5cm}
 \setlength{\unitlength}{.3mm} 
\end{example}

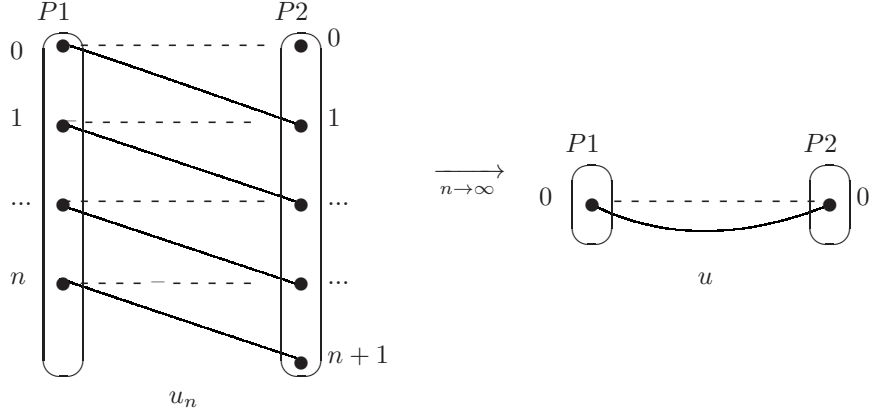
\begin{center}
\begin{picture}(350,120)
%\par

\put(20,59){- - - - - - - - - - - - } 
\put(20,28){- - - - - -- - - - - - } 
\put(20,118){- - - - - - - - - - - - } 
\put(20,89){-- - - - - - - - - - - } 
%\par \textcolor{blue}{\put(30,60){\line(1,0){90}}}
%\par \textcolor{blue}{\put(20,30){\line(1,0){90}}}
%\par \textcolor{blue}{\put(20,120){\line(1,0){90}}}
%\par \textcolor{blue}{\put(20,90){\line(1,0){90}}}
\par \textcolor{black}{{}\qbezier(19,32) (18,32)(108,02) }
\par \textcolor{black}{{} \qbezier(15,60) (15,60)(105,30)}
\par \textcolor{black}{{} \qbezier(15,90) (15,90)(105,60) }
\par \textcolor{black}{{} \qbezier(13,120) (32,114)(100,90)}
\par \put(10, 60){\oval(15, 130)} \put(100, 60){\oval(15, 130)}
\par \put(10,60){\circle*{5}} \put(10,30){\circle*{5}} \put(10,90){\circle*{5}}
\put(10,120){\circle*{5}} \put(100,90){\circle*{5}} \put(100,120){\circle*{5}}
\put(100,00){\circle*{5}} \put(100,60){\circle*{5}} \put(100,30){\circle*{5}}
\par \put(0,130){$P1$} \put(90,130){$P2$}
\par \put(-10,115){$0$} \put(-10,90){$1$} \put(-10,60){$...$}
\put(-10,30){$n$}
\par \put(110,0){$n+1$} \put(110,60){$...$} \put(110,30){$...$}
\put(110,90){$1$} \put(110,120){$0$}
\par \put(50, -15){$u_{n}$}
\par \put(150,70){$\xrightarrow[n\to\infty]{}$}
\par %\par \textcolor{blue}{\put(215,60){\line(1,0){90}}}
\put(215,59){ - - - - - - - - - - - - } 
\par \textcolor{black}{{} \qbezier(210,60) (250,40)(300,60)}
\par \put(250, 30){$u$}
\par \put(210, 60){\oval(15, 30)} \put(300, 60){\oval(15, 30)}
\par \put(210,60){\circle*{5}}
\par \put(300,60){\circle*{5}}
\par \put(200,80){$P1$} \put(290,80){$P2$}
\par \put(190,60){$0$}
\par \put(310,60){$0$}
\end{picture} 
\end{center}

\vspace{0.5cm}
 
%\par
%\end{center}

The idea is that when $n$ is large, with high probability the players
will receive signals far from $0$ and $n$. These signals convey
very little information to the players and only differ for high-order
beliefs.

Notice that the signal of each player is $\frac{1}{n+1}$-independent
from the state of the world. The Lemma implies that $d\left(u_{n},u\right)\leq\frac{2}{n+1}\rightarrow0$.
In particular, the information structures $u_{n}$ converge to the
structure $u$, where no player receive any information.

\section{Value-based topology\label{sec:Topology}}

The two previous sections discussed the quantitative aspect of the
value-based distance. Now, we analyze its qualitative aspect: the
topological information contained in the distance. 

\subsection{Universal space of information structures\label{subsec:Universal-space-of}}

We begin with recalling relevant facts about the universal type spaces.
A (Mertens-Zamir) \emph{universal type space} is a pair of compact
metric spaces $\Theta_{1}$ and $\Theta_{2}$ with respect to the
following property: for each $i$, there exists a homeomorphism $\phi_{i}:\Theta_{i}\rightarrow\Delta\left(K\times\Theta_{-i}\right)$.
The universal type space is unique (up to homeomorphism). The spaces
$\Theta_{i}$ are constructed as sets of coherent sequences of finite
hierarchies of beliefs. The topology on the universal type space is
also is referred to as the product topology; the latter name comes
from the fact that it can be defined as the Tichonof's topology on
the sequences of finite hierarchies.

For each Harsanyi type space (even uncountable), there is a belief-preserving
mapping that uniquely maps the types into the associated hierarchies.
And so, if $u\in\Delta\left(K\times C\times D\right)$ is an information
structure, the signals $c$ and $d$ are mapped into hierarchies $\tilde{c}\in\Theta_{1}$
and $\tilde{d}\in\Theta_{2}$. The mapping induces a probability distribution
$\tilde{u}\in\Delta\Theta_{1}$ over the hierarchies of beliefs of
player 1 (and an analogous object for player 2). Let $\Pi_{0}=\left\{ \tilde{u}:u\in\CU\right\} $
be the space of all such distributions. The closure of $\Pi_{0}$
(in the weak topology) is denoted as $\Pi$. Using the Mertens-Zamir's
terminology, $\Pi$ is the set of consistent (Borel) probabilities
over the universal type space, and $\Pi_{0}$ is the subset of elements
of $\Pi$ with at most countable support. The space $\Pi$ is compact
metric and $\Pi_{0}$ is dense in $\Pi$ (see corollary III.2.3 and
theorem III.3.1 in \citet{mertens_sorin_zamir_2015}).

On the other hand, each $P\in\Pi$ can be treated as a common prior
type space with types in sets $\Theta_{i}$. The common prior is generated
as follows. First, a type $\theta_{1}$ of player 1 is drawn from
distribution $P$. Then, a state $k$ and a type $\theta_{2}$ of
player $2$ is drawn from distribution $\phi_{1}\left(\theta_{1}\right)$.
For each zero-sum game $g\in\CG$, Proposition III.4.2 in \citet{mertens_sorin_zamir_2015}
shows that one can find its value on $P$ that we denote as $\val\left(P,g\right)$.
In particular, we can define the value based distance (\ref{eq:ZS distance})
over elements of $\Pi$.

There are two fundamental observations. First, by Proposition III.4.4
in \citet{mertens_sorin_zamir_2015}, 
\[
\val(u,g)=\val(\tilde{u},g)\text{ for each }u\in\CU.
\]
In particular, the value on information structure $u$ depends only
on the induced distribution over hierarchies of beliefs. Second, by
lemma 41 in \citet{gossner_value_2001}), the value separates the
elements of $\Pi$: we have 
\[
d\left(P,Q\right)>0\text{ for any }P,Q\in\Pi,P\neq Q.
\]
The two fundamental observations imply that there is a bijection between
the set of $\Pi_{0}$ and the set $\CU^{*}$ of equivalence classes
of information structures.

\subsection{The weak vs. the value-based topology}

The next result says the spaces $\Pi_{0}$ and the set $\CU^{*}$
have the same topologies. 
\begin{thm}
\label{Thm: weak topology} A sequence $(u_{n})$ in $\CU^{*}$ converges
to $u\in\CU^{*}$ for the value-based distance if and only if the
sequence $(\tilde{u}_{n})$ converges weakly to $\tilde{u}$ in $\Pi_{0}$.
%The mapping from $\CU^{*}$ to $\Pi_{0}$ which associates $\tilde u$ to $u$
%is a homeomorphism between $\Pi_{0}$ with the weak topology and $\CU^{*}$
%with the value-based topology. 
\end{thm}

The result says that a convergence in value-based topology to a countable
structure is equivalent to the convergence in distribution of finite-order
hierarchies of beliefs. Informally, at least around countable structures,
the higher-order beliefs have diminishing importance.

We describe the idea of the proof when $u$ is finite. In such a case,
we surround the hierarchies $\tilde{c}$ for $c\in C$ by sufficiently
small and disjoint neighborhoods, so that all hierarchies in the neighborhood
of $\tilde{c}$ have similar beliefs about the state and the opponent.
We do alike for the other player. The weak convergence ensures that
the converging structures assign large probability to the neighborhoods.
Finally, we show that any information about players hierarchy beyond
the neighborhood to which it belongs is almost conditionally independent
(in the sense of Section \ref{sec:Value-of-additional}) from the
information about the state and the opponents neighborhoods. This
allows us to utilize Proposition \ref{prop: joint info} (or, more
precisely, a part of its proof) to show that only the information
about neighborhoods matter, and the latter is similar to the information
in the limit structure $u$. If $u$ is countable, we also show that
it can be appropriately approximated by finite structures.

A reader may find such a result surprising, given the message of the
strategic discontinuities literature (\citet{rubinstein_electronic_1989},
\citet{dekel_topologies_2006}, \citet{weinstein2007structure}, \citet{ely_critical_2011},
etc.). In that literature, the convergence of finite-order hierarchies
does not imply strategic convergence even around finite structures
(for instance, see Example \ref{ex:Rubinstein} below). There are
multiple ways in which our setting differs. First, we rely on ex ante
equilibrium concept, rather than interim rationalizability. We are
also interested in the payoff comparison rather than the behavior.
Second, we restrict attention to common prior type spaces. Finally,
we restrict attention to zero-sum games.

We believe that the last restriction makes the key difference. As
we explain in Section \ref{subsec:Payoff distance NZS} below, the
ex ante focus and payoff comparison but without restriction to zero-sum
games lead to a topology that is significantly finer than the weak
topology. The role of common prior is less clear. On one hand, \citet{lipman:03}
imply that, at least from the interim perspective, common prior does
not generate significant restrictions on finite order hierarchies.
On the other hand, we rely on the ex ante perspective, and common
prior is definitely important for Proposition \ref{prop: joint info}
that plays an important role in the proof.

Finally, we want to emphasize that Theorem \ref{Thm: weak topology}
does not mean that the higher-order beliefs do not matter at all.
In fact, this result does not hold anymore if the limit $u$ is an
uncountable information structure, as illustrated in Section \ref{sec:Compactness}.
%it is correct but not completely obvious. footnote about the fact that the structures in the sequence might be uncountableWe
%investigate this issue in detail in Section \ref{sec:Compactness}.

\subsubsection{Approximate knowledge and approximate common knowledge\label{subsec:Approximate-knowledge}}

Here, we describe an important application of Theorem \ref{Thm: weak topology}.
Recall that one of the initial impulses of the literature on the higher-order
beliefs was the realization of the difference between the approximate
knowledge and the approximate common knowledge. Here, we are going
to show that, at least when it comes to the knowledge of the payoff
state, the difference is not important for zero-sum games.

An information structure $u\in\Delta\left(K\times C\times D\right)$
exhibits $\varepsilon$-knowledge of the state if there is a mapping
$\kappa:C\cup D\rightarrow K$ such that 
\[
u\Big(\left\{ u(\{k=\kappa(c)\}|c)\geq1-\varepsilon\right\} \Big)\geq1-\varepsilon\text{ and }u\Big(\left\{ u(\{k=\kappa(d)\}|d)\geq1-\varepsilon\right\} \Big)\geq1-\varepsilon.
\]
In other words, the probability that any of the player player assigns
at least $1-\varepsilon$ to some state is at least $1-\varepsilon$. 
\begin{prop}
\label{prop: approximate knowledge}Suppose that $u$ exhibits $\varepsilon$-knowledge
of the state and that $v\in\Delta\left(K\times K_{C}\times K_{D}\right)$,
where $K_{C}=K_{D}=K$ and $\marg_{K}v=\marg_{K}u$, and $v\left(k=k_{C}=k_{D}\right)=1$.
(In other words, $v$ is a common knowledge structure with the only
information about the state.) Then, 
\[
d\left(u,v\right)\leq20\varepsilon.
\]
\end{prop}

Thus, approximate knowledge structures are close to common knowledge
structures. The convergence of approximate knowledge type spaces to
the approximate common knowledge is a consequence of Theorem \ref{Thm: weak topology}.
The metric bound stated in the Proposition requires a separate (simple)
proof based on Proposition \ref{prop: joint info}. 
\begin{example}
\label{ex:Rubinstein}Consider a Rubinstein email-game information
structures $u^{p,\alpha}$ defined as follows: The state is $k=1$
with probability $p$ and $k=0$ with the remaining probability. If
the state is 0, both players receive signal 0. If the state is 1,
player 1 learns that and sends a message to the opponent, who after
receiving it, immediately sends it back. The message travels back
and forth. Each round, the message may get lost before reaching the
target with probability $\alpha$, in which case the process stops.
The signal of the player is the number of the times she sends a message
away.

Player 1 always knows the state. Player 2's first order belief attach
the probability of at least $\frac{1}{1+\alpha\frac{p}{1-p}}$ to
one of the states. Proposition \ref{prop: approximate knowledge}
implies that, when $\alpha\rightarrow0$, the Rubinstein's information
structures converge to the common knowledge of the state. 
\end{example}

\section{Large space of information structures\label{sec:Compactness}}

In this section, we show that the value-based distance does not have
a compact completion, or more precisely, it is not totally bounded
on $\CU^{*}$. Informally, the space of information structures is
large: it cannot be approximated by finitely many structures. We use
this result to provide a negative answer to a problem posed by J.F.
Mertens. 

\subsection{$(\CU^{*},d)$ is not totally bounded}

To focus attention, we assume that $K=\left\{ 0,1\right\} $. This
is without loss of generality, as our negative result for $K=\left\{ 0,1\right\} $
clearly implies a negative result for larger $K$. 

The next result finds $\varepsilon>0$ and an infinite sequence of
(finite) information structures such that all of them are at least
$\varepsilon$-away from each other in the value-based distance.
\begin{thm}
\label{thm3} There exists $\varepsilon>0$, even $N<\infty$, and
a Markov chain of distribution $\mu$ over sequences 
\[
k,c_{1},d_{1},c_{2},d_{2},c_{3},...
\]
where $C_{l}=D_{l}=\left\{ 1,...,N\right\} $ and $c_{l}\in C_{l},d_{l}\in D_{l}$
for each $l\geq1$, such that, if we define sets of signals $C^{l}=\prod_{p=1}^{l}C_{p},D^{l}=\prod_{p=1}^{l}D_{p}$
and information structures 
\[
u^{l}=\marg_{K\times C^{l}\times D^{l}}\mu,
\]
then, for any $p\neq l$, 
\[
d\left(u^{l},u^{p}\right)>\varepsilon.
\]
\end{thm}

Recall that a consistent probability over the universal type space
is a probability distribution over hierarchies of beliefs that can
be induced by some (potentially uncountable) common prior type space,
and that the set $\Pi$ of consistent probabilities over the universal
belief space is compact under the topology of the weak convergence.
Theorem \ref{thm3} implies that $(\Pi_{0},d)$ and therefore $(\Pi,d)$
are not totally bounded metric spaces, so that $(\Pi,d)$ is not a
compact metric space. It follows that the value-based topology on
$\Pi$ is different from the weak topology, even if these topologies
coincide on $\Pi_{0}$ by Theorem \ref{Thm: weak topology}. (We elaborate
further on this point in section \ref{subsec:Weak-zero-sum-topology}.)

The proof, with an exception of one step that we describe below, is
constructive. The first part of the theorem describes the properties
of the construction. We define a Markov chain that starts with a state
followed by alternating signals for each player. In structure $u^{n}$,
player 1 observes signals $(c_{1},c_{2},....,c_{n})$ and player 2
observes $(d_{1},d_{2},....,d_{n})$. Thus, the sequence of structures
$u^{n}$ can be understood as fragments of a larger information structure,
where progressively, more and more information is revealed to each
player. The Theorem shows that the larger structure is not the limit
of its fragments in the value-based distance. In particular, there
is no analog of the martingale convergence theorem for the value-based
distance for such sequences.

This has to be contrasted with two other settings, where the martinagle
convergence holds. First, in the 1-player case, any sequence of information
structures in which the player is receiving more and more signals
converges for the distance $d_{1}$ (which can be proved using formula
\eqref{eq:d1weak}). Second, using martingale convergence theorems
and arguing in the proof of the denseness of $\Pi_{0}$ in $\Pi$
for the weak topology (see e.g. Thm III.3.1 and Lemma III.3.2 in \citet{mertens_sorin_zamir_2015}),
one shows that $\tilde{u}^{n}$ converges $\tilde{\mu}$ in $\Pi$
in the weak topology.

The Markov property means that (a) the state is independent from all
players information conditionally on $c^{1}$, and (b) each new piece
of information is independent from the previous pieces of information
conditionally on the most recent information of the other player.
As we show in the Appendix \ref{subsec:Higher-order-beliefs}, this
ensures that the $n$-th level hierarchy of beliefs of any type in
structure $u^{n}$ is preserved by all consistent types in structures
$u^{m}$ for $m\geq n$. Theorem \ref{thm3} exhibits therefore a
situation in which higher-order beliefs do not have diminishing importance
for the value-based distance. In particular, it shows that the knowledge
of the $n$-th level hierarchy of beliefs for any arbitrarily high
$n$ is not sufficient to play $\varepsilon$-optimally in any finite
zero-sum game $g\in\CG$.

%As a corollary, we obtain that
%the value-based distance is not continuous with respect to the product
%topology on the uniersal type space on $K$.

\subsection{Comments on the proof}

We briefly sketch the main ideas behind the proof. We fix $\alpha<\frac{1}{25}$.
We show that we can find even $N$ high enough and a set $S\subseteq\left\{ 1,...,N\right\} ^{2}$
with certain mixing properties: 
\begin{align*}
\left|\left\{ j:\left(i,j\right)\in S\right\} \right| & =\frac{N}{2},\text{ for each }i,\\
\left|\left\{ i:\left(i,j\right)\in S\right\} \right| & \sim\frac{N}{2},\text{ for each }j,\\
\left|\left\{ i:\left(i,j\right),\left(i,j^{\prime}\right)\in S\right\} \right| & \sim\frac{N}{4},\text{ for each }j,j{}^{\prime},\\
\left|\left\{ i:\left(i,j\right),\left(i,j^{\prime}\right),\left(l,i\right)\in S\right\} \right| & \sim\frac{N}{8},\text{ for each }j,j{}^{\prime},l,
\end{align*}
etc. The ``$\sim$'' means that the left-hand side is within $\alpha$-related
distance to the right-hand side. Altogether, there are 8 properties
of this sort (see Appendix \ref{subsec:Existence-of-Markov chain}
for details.) The properties essentially mean that various sections
of $S$ are ``uncorrelated'' with each other.

Although it might be possible to directly construct $S$ with the
required properties, we are unable to do so. Instead, we show the
existence of set $S$ using the probabilistic method of P. Erd\H{o}s
(for a general overview of the method, see \citet{alon_probabilistic_2008}).
%(We are aware of only one application of the probabilistic method
%in the economic theory in SHMAYA CITE????. 
Suppose that the sets $S\left(i\right)$ for $i=1,...,N$ are chosen
independently and uniformly from all $\frac{N}{2}$-element subsets
of $\left\{ 1,...,N\right\} $. We show that if $N$ is sufficiently
large, then the set $S=\left\{ \left(i,j\right):j\in S\left(i\right)\right\} $
satisfies the required properties with positive probability, proving
thus that a set satisfying these properties exists. Our method of
the proof is not particularly careful about the optimal $N$ and a
rough estimate suggests that it needs $N\geq10^{8}$.

Given $S$, we construct the Markov chain of distribution $\mu$.
First, $c_{1}$ is chosen from the uniform distribution on $\left\{ 1,...,N\right\} $.
Next, we choose $k=1$ with probability $\frac{c_{1}}{N+1}$ and $k=0$
with the remaining probability. Next, inductively, for each $l\geq1$,
we choose 
\begin{itemize}
\item $d_{l}$ uniformly from set $S\left(c_{l}\right)=\left\{ j:\left(c_{l},j\right)\in S\right\} $
and conditionally independently from $k,...,d_{l-1}$ given $c_{l}$,
and 
\item $c_{l+1}$ uniformly from set $S\left(d_{l}\right)$ and conditionally
independently from $k,...,c_{l}$ given $d_{l}$. 
\end{itemize}
Thus, the Markov chain becomes time-homogenuous after the choice of
state $k$ and $c_{1}$. This finishes the construction of the information
structures.

To provide a lower bound on the distance between different information
structures, we construct a sequence of games. In game $g^{n}$, player
1 is supposed to reveal the first $n$ pieces of her information;
player 2 reveals the first $n-1$ pieces. The payoffs are such that
it is a dominant strategy for player 1 to precisely reveal his first
order belief about the state, which amounts to report truthfully $c_{1}$.
Further than that, we verify whether the sequence of reports 
\[
\left(\hat{c}_{1},\hat{d}_{1},...,\hat{c}_{n-1},\hat{d}_{n-1},\hat{c}_{n}\right)
\]
belongs to the support of the distribution of the Markov chain $\mu$.
If it does, then player 1 receives payoff $\varepsilon\sim\frac{1}{10\left(N+1\right)^{2}}$.
If it doesn't, we identify the first report in the sequence that deviates
from the support. The responsible player is punished with payoff $-5\varepsilon$
(and the opponent receives $5\varepsilon$).

The payoffs and the mixing properties of matrix $S$ ensure that players
have incentives to report their information truthfully. We check it
formally, and we show that that if $l>p$, then 
\[
d\left(u^{l},u^{p}\right)\geq\val\left(u^{l},g^{p+1}\right)-\val\left(u^{p},g^{p+1}\right)\geq2\varepsilon.
\]

\subsection{An open problem raised by Mertens}

\label{subsec:Mertens}

%Recall that a consistent probability over the universal
%type space is a probability distribution over hierarchies of beliefs
%that can be induced by the (potentially uncountable) common prior type
%space. Let $\Pi$ be the set of consistent probabilities over the
%Mertens-Zamir universal belief space. For each information structure
%$u$, let $\tilde{u}\in\Pi$ denote the induced consistent probability. 
%The set $\Pi$ is compact when endowed with the weak topology.

Because the $n$-th level hierarchies of beliefs become constant as
we move along the sequence $u^{n}$, it must be\footnote{ As mentioned earlier, this convergence can also be proved by using
martingale convergence theorems. } that the sequence $\tilde{u}^{n}$ converges weakly in $\Pi$ to
the limit 
\[
\tilde{u}^{n}\rightarrow\tilde{\mu}.
\]
(The limit is the consistent probability obtained from the type space
over infinite sequences with common prior $\mu$.) Despite the convergence,
our Theorem shows that 
\[
\limsup_{n}\sup_{g\in\CG}\left|\val\left(\mu,g\right)-\val\left(u^{n},g\right)\right|=\limsup_{n}d(\tilde{u}^{n},\tilde{\mu})\geq\varepsilon.
\]
In particular, the family of functions $(u\mapsto\val(u,g))_{g\in\CG}$
is not equicontinuous on $\Pi$ for the weak topology.

This answers negatively the second of the three problems\footnote{Problem 1 asked for the convergence of the value functions $(v_{\lambda})_{\lambda}$
and $(v_{n})_{n}$ in a general zero-sum repeated game with finitely
many states, actions and signals, and was disproved in \citet{ziliotto_zero-sum_2016}.
Problem 3 asks if the existence of a uniform value follows from the
uniform convergence of $(v_{\lambda})$, and was disproved in \citet{lehrer_discounting_1994}
for 1-player games, see also \citet{monderer_asymptotic_1993}.} posed by \citet{mertens_repeated_1986} in his Repeated Games survey
from ICM: ``This equicontinuity or Lispchitz property character is
crucial in many papers...''. Precisely, if this family was equicontinuous,
a direct application of Ascoli's theorem would have implied that the
family $(u\mapsto\val(u,g))_{g\in\CG}$ is totally bounded in the
space of bounded functions over $\Pi$ endowed with the uniform norm.

The importance of the Mertens question comes from its application
to limit theorems in the repeated games. Consider a general zero-sum
repeated game (stochastic game, with incomplete information, and imperfect
monitoring), given by a transition $q:K\times I\times J\longrightarrow\Delta(K\times C\times D)$,
a payoff function $g:K\times I\times J\longrightarrow[-1,1]$ and
an initial probability $u_{0}$ in $\Delta(K\times C\times D)$, where
$K$, $I$, $J$, $C$ and $D$ are finite sets. Before stage 1, an
initial state $k_{1}$ in $K$ and initial private signals $c_{1}\in C$
for player 1, and $d_{1}\in D$ for player 2, are selected according
to $u_{0}$. Then at each stage $t$, simultaneously player 1 chooses
an action $i_{t}$ in $I$ and player 2 chooses and action $j-t$
in $J$, the stage payoff is $g(k_{t},i_{t},j_{t})$, an element $(k_{t+1},c_{t+1},d_{t+1})$
is selected according to $q(k_{t},i_{t},j_{t})$, the new state is
$k_{t+1}$, player 1 receives the signal $c_{t+1}$, player 2 the
signal $d_{t+1}$, and the play proceeds to stage $t+1$.

An appropriate state variable is $u_{t}\in\Delta\left(K\times\left(I^{t-1}\times C^{t}\right)\times\left(J^{t-1}\times D^{t}\right)\right)$,
representing the current state in $K$ and the finite sequence of
private actions and signals previously received by each player. As
a consequence, a recursive formula can be explicitly written as follows:
for all discount factors $\delta<1$ and all $u\in\Delta\left(K\times C_{0}\times D_{0}\right)$,
\begin{eqnarray*}
v_{\delta}(u;g,q) & = & \max_{\sigma_{1}:C_{0}\rightarrow\Delta(I)}\min_{\sigma_{2}:D_{0}\rightarrow\Delta(J)}\left(1-\delta\right)\gamma_{u,g}(\sigma_{1},\sigma_{2})+\delta v_{\delta}(F(u,\sigma_{1},\sigma_{2};q);g,q),\\
 & = & \min_{\sigma_{2}:D_{0}\rightarrow\Delta(J)}\max_{\sigma_{1}:C_{0}\rightarrow\Delta(I)}\left(1-\delta\right)\gamma_{u,g}(\sigma_{1},\sigma_{2})+\delta v_{\delta}(F(u,\sigma_{1},\sigma_{2};q);g,q),
\end{eqnarray*}
where $\sigma_{1},\sigma_{2}$ are strategies in the game, $\gamma_{u,g}(\sigma_{1},\sigma_{2})$
is the stage game payoff, and $F(u,\sigma_{1},\sigma_{2};q)\in\Delta\left(K\times(C_{0}\times I\times C)\times(D_{0}\times J\times D)\right)$
is the information structure obtained tomorrow, given today's strategies,
the state transition and signal function $q$. Formally, 
\[
F\left(u,\sigma_{1},\sigma_{2};q\right)(k,c_{0},i,c,d_{0},j,d)=\sum_{k^{\prime}\in K}u\left(k^{\prime},c_{0},d_{0}\right)\sigma_{1}\left(i|c_{0}\right)\sigma_{2}\left(j|d_{0}\right)q\left(k|k^{\prime},i,j\right).
\]

\begin{prop}
\label{prop_precompact_sto} Suppose that the family of information
structures $\CU_{0}\subseteq\CU$ is totally bounded for the value-based
distance $d$. Suppose that for some stochastic game $\left(g,q\right)$,
any strategies $\sigma_{1},\sigma_{2}$, we have 
\[
u\in\CU_{0}\Longrightarrow F\left(u,\sigma_{1},\sigma_{2};q\right)\in\CU_{0}.
\]
Then, the family $(v_{\delta}\left(.;g,q\right))_{\delta\in[0,1)}$
is totally bounded in the space of bounded functions on $\CU_{0}$ endowed
with the uniform norm. 
\end{prop}

\begin{proof}
Let $\CU_{0}^{*}\subset\CU^{*}$ be the set of equivalence classes
of information structures in $\CU_{0}$. The metric space $(\CU_{0}^{*},d)$
being totally bounded, its completion $(\overline{\CU}_{0}^{*},d)$
is a compact metric space. The functions $v_{\delta}$ are, by construction,
1-Lipschitz from $(\CU_{0}^{*},d)$ to $[-1,1]$. Hence any function
$v_{\delta}$ admits a unique $1$-Lipschitz extension on $\overline{\CU}_{0}^{*}$.
Ascoli's theorem implies that the family $(v_{\delta})_{\delta\in[0,1)}$
is totally bounded in the space of bounded functions on $(\overline{\CU}_{0},d)$
endowed with the uniform norm, which implies the result. 
\end{proof}
The above result might seem quite weak, but it is the first necessary
step for proving that $v_{\delta}$ converges uniformly when $\delta$
goes to $1$. This condition was shown to be sufficient for general
Markov Decision Problems in \citet{renault2011}, and his main theorem
(together with the Tauberian theorem of \citet{lehrer_sorin}) implies
quite directly the following corollary. 
\begin{cor}
In addition to the assumptions of Proposition \ref{prop_precompact_sto},
assume that the transition $q$ of the stochastic game does not depend
on the action of player $2$, then the family $v_{\delta}$ converges
uniformly to a limit $v$ on $\CU_{0}$ as $\delta$ goes to $1$. 
\end{cor}

%\begin{proof}
%Since the transition of the game depends only on the action of player $1$, $v_{\delta}$ can be interpreted as the value of a Markovian Decision Problem (as in Renault 2012), and therefore the result follows directly from Proposition \ref{prop_precompact_sto},  Theorem ?? in Renault 2011 and the Tauberian theorem of Lehrer-Sorin ??.
%\end{proof}
The existence of a limit value attracted a lot of attention since
the first results by \citet{aumann_repeated_1995} and \citet{mertens1971value}
for repeated games, and by \citet{bewley1976asymptotic} for stochastic
games. Once the fact that the family of value functions is totally
bounded is established, the existence of the limit value is typically
obtained by showing that there is at most one accumulation point of
the family $(v_{\delta})$, e.g. by showing that any accumulation
point satisfies a system of variational inequalities admitting at
most one solution (see e.g. the survey \citet{laraki2015advances}
and footnote \ref{foot_zerosum} for related works).

Even if the answer to Mertens' problem is negative, identifying totally
bounded subsets of $(\CU^{*},d)$ might have applications for limit
theorems for repeated games for which the hypothesis of Proposition
\ref{prop_precompact_sto} applies. The next section illustrate with
our notation several known such subsets.

\subsection{Totally bounded subsets\label{subsec:Compact-subsets}}

In the 1-player case, the characterization \ref{eq:d1weak} implies
that the set of equivalence classes in $\CU$ inducing the same distribution
of first order belief of player $1$, equipped with the distance $d_{1}$
can be identified with the subset of $(\Delta(\Delta(K)),\tilde{d}_{1})$
of probabilities with countable support, where $\tilde{d}_{1}$ is
the distance given by right-hand side of \eqref{eq:d1weak}, which
metricizes the weak topology on $\Delta(\Delta(K))$. The latter being
weakly compact, we deduce that $(\CU,d_{1})$ is totally bounded.

For 2 players, every particular subspace $\CU_{i}$ for $i=1,...,4$
listed below has been shown to be totally bounded, each time by identifying
$\CU_{i}^{*}$ with a subset of a weakly compact space of probabilities.
We only mention the relevant variables used for the identification
and the corresponding space of probabilities.

%If $\CU'$ is a subset of $\CU^{*}$, we denote by $\overline{\CU'}$ the closure of $\CU'$ in $\overline{\CU}$ :

$\bullet$ Set $\CU_{1}$ of information structures where both players
receive the same signal: $\CU_{1}^{*}$ can be identified with $\Delta_{c}(\Delta(K))$\footnote{ We use the notation $\Delta_{c}(X)$ for the probabilities with at
most countable support on $X$}. Here given $u$ in $\CU_{1}$, what matters is the induced law on
the common a posteriori of the players on $K$. Another characterization
of $d(u,v)$ has been obtained in \citet{renault_venel_2016}. let
$D_{1}$ be the subset of 1-Lipschitz functions from $\Delta(K)$
to $\R$ satisfying 
\[
\forall p,q\in\Delta(K),\forall a,b\geq0,\;af(p)-bf(q)\leq\|ap-bq\|_{1}.
\]
We have : 
\[
\forall u,v\in\CU_{1},\;d(u,v)=\sup_{f\in D_{1}}\left(\int_{p\in\Delta(K)}f(p)d{\tilde{u}}(p)-\int_{p\in\Delta(K)}f(p)d{\tilde{v}}(p)\right).
\]

$\bullet$ Set $\CU_{2}$ of information structures $u\in\Delta(K\times C\times D)$
where player 1 knows the signal of player 2, i.e. there exists a map
$\kappa:C\rightarrow D$ such that $u(\{d=\kappa(c)\})=1$: $\CU_{2}^{*}$
can be identified with $\Delta_{c}(\Delta_{c}(\Delta(K)))$ (see \citet{mertens_repeated_1986},
\citet{gensbittel2014}). Given $u\in\CU_{2}$, the relevant variable
is the induced distribution of the second order beliefs of player
$2$.

$\bullet$ Set $\CU_{3}$ of independent information structures: $\CU_{3}$
is the set of $u\in\CU$ with conditionally independent information.
$\CU_{3}^{*}$ can be identified with $\Delta_{c}(\Delta(K)\times\Delta(K))$.
Given $u$ in $\CU_{3}$, the relevant variable is the pair of first
order beliefs of the players.

\section{Relation to other topologies on $\CU$\label{sec:Relation-to-other}}

\subsection{Payoff-based distance with non-zero-sum games\label{subsec:Payoff distance NZS}}

It seems natural to ask about a version of distance (\ref{eq:ZS distance})
where we take supremum over all bounded games, not necessarily zero-sum.
In this subsection, we are going to show that such an approach cannot
be useful because so defined measure is equal to its maximum possible
value for almost every pair of information structures.

Formally, a \emph{non-zero sum payoff function} is a map $g:K\times I\times J\rightarrow\left[-1,1\right]^{2}$
where $I,J$ are finite sets. Let $\text{Eq}\left(u,g\right)\subseteq\R^{2}$
be the set of Bayesian Nash payoffs in game $g$ on information structure
$u$. Assume that the space $\R^{2}$ is equipped with the maximum
norm $d_{\max}\left(x,y\right)=\max_{i=1,2}\left|x_{i}-y_{i}\right|$
and the space of compact subsets of $\R^{2}$ with the induced Hausdorff
distance $d_{\max}^{H}$. Let 
\[
d_{NZS}\left(u,v\right)=\sup_{g\text{ is a non-zero-sum payoff function}}d_{\max}^{H}\left(\text{Eq}\left(u,g\right),\text{Eq}\left(v,g\right)\right).
\]
Then, clearly as in our original definition (\ref{eq:ZS distance}),
$0\leq d_{NZS}\left(u,v\right)\leq2.$

Contrary to the value in the zero-sum game, the Bayesian Nash Equilibrium
payoffs on information structure $u$ cannot be factorized through
the distribution $\tilde{u}\in\Delta(\Theta^{1})$ over the hierarchies
of beliefs induced by $u$. For this reason, we only restrict our
analysis to information structures that are non-redundant, or equivalently
information structures induced by a coherent probability with countable
support in $\Pi_{0}$. That is not because we think that the redundant
information is not relevant (it is relevant for Bayesian Nash equilibrium,
see multiple examples in \citet{dekel_topologies_2006}, \citet{ElyPeski06}),
but because the dependence of the BNE on the redundant information
is not yet well-understood (see \citet{sadzik2008beliefs})\footnote{An alternative approach would be to take an equilibrium solution concept
that can be factorized through the hierarchies of beliefs. An example
is Bayes Correlated Equilibrium from \citet{bergemann_bayes_2015}.}.

Let $u\in\Delta\left(K\times C\times D\right)$ be an information
structure. A subset $A\subseteq K\times I\times J$ is a \emph{proper
common knowledge event} if $u\left(A\right)\in\left(0,1\right)$ and
for each signal $s\in C\cup D$, $u\left(A|s\right)\in\left\{ 0,1\right\} $.
An information structure is \emph{simple} if it does not have a proper
common knowledge component. Each finite non-redundant information
structure $u$ has a representation as a convex combination of (non-redundant)
simple information structures $u=\sum_{\alpha}p_{\alpha}u_{\alpha}$,
where $\sum p_{\alpha}=1,p_{\alpha}\geq0$, and $p_{\alpha}>0$ for
at most finitely many $\alpha$. 
\begin{thm}
\label{thm:NZS distance}Suppose that $u,v$ are finite and non-redundant
information structures. If $u$ and $v$ are simple, then 
\begin{align*}
d_{NZS}\left(u,v\right) & =\begin{cases}
0, & \text{if }\tilde{u}=\tilde{v},\\
2 & \text{otherwise. }
\end{cases}
\end{align*}
More generally, suppose that $u=\sum p_{\alpha}u_{\alpha}$ and $v=\sum q_{\alpha}v_{\alpha}$
are the decompositions into simple information structures. We can
always choose the decompositions so that $\tilde{u}_{\alpha}=\tilde{v}_{\alpha}$
for each $\alpha$. Then, 
\[
d_{NZS}\left(u,v\right)=\sum_{\alpha}\left|p_{\alpha}-q_{\alpha}\right|.
\]
\end{thm}

The distance between the two non-redundant simple information structures
is binary, either 0 if the information structures are equivalent,
or 2 if they are not. In particular, the distance between most of
the simple information structures is equal to its maximum possible
value 2. The distance $d_{NZS}$ between two non-redundant, but not
necessarily simple information structures depends on how similar is
their decomposition into the simple components.

The proof in the case of two non-redundant and simple structures $u$
and $v$ is very simple. Let $\tilde{u}\neq\tilde{v}$. First, it
is well-known that there exist a finite game $g:K\times I\times J\rightarrow[-1,1]^{2}$
in which each type of player 1 in the support of $\tilde{u}$ and
$\tilde{v}$ reports her hierarchy of beliefs as the unique rationalizable
action. Second, Lemma III.2.7 in \citet{mertens_sorin_zamir_2015}
(or corollary 4.7 in \citet{mertzam:85}) shows that the supports
of distributions $\tilde{u}$ and $\tilde{v}$ must be disjoint (it
is also a consequence of the result by \citet{samet1998iterated}).
Thus, we can construct a game, in which, additionally to the first
game, player 2 chooses between two actions $\left\{ u,v\right\} $
which influences her payoff only, in such a way that it is optimal
for her to match the information structure to which player 1's reported
type belongs. Finally, we multiply the so obtained game by $\varepsilon>0$
and construct a new game, in which, additionally, player 1 receives
payoff $1-\varepsilon$ if player 2 chooses $u$ and a payoff of $-1+\varepsilon$
if player 2 chooses $v$. Hence, the payoff distance between the two
information structures is at least $2-\varepsilon$, where $\varepsilon$
is arbitrary small.

\subsection{Pointwise topology\label{subsec:Weak-zero-sum-topology}}

An alternative way to define the topology (but not the distance) on
the space of information structures would be through the convergence
of values. Say that a sequence of information structures $u_{n}$
converges to $u$ pointwise if for all payoff structures $g\in\CG$,
$\lim_{n\rightarrow\infty}\val(u_{n},g)=\val(u,g)$. Note that pointwise
convergence is also well-defined on $\Pi$.

The topology of pointwise convergence is the weakest topology that
makes the value mappings continuous. This property is not as easy
to interpret and it does not seem to be as useful as the value-based
metric induced topology. The next result provides its characterization. 
\begin{thm}
\label{thm5} Sequence $u_{n}$ of information structures converges
pointwise to $u$ if and only if $d\left(u_{n},u\right)\rightarrow0$.
In other words, the topology of pointwise convergence and the value-based
topology coincide on $\CU^{*}$ and they are homeomorphic to the weak
topology on $\Pi_{0}$. \\
 The topology of weak convergence on $\Pi$ coincides with the topology
of pointwise convergence. 
\end{thm}

\begin{proof}
The Theorem follows from the observations made in Section \ref{subsec:Universal-space-of}.
On one hand, lemma 41 of \citet{gossner_value_2001} implies that
the weak topology on $\Pi$ is coarser than the topology of weak convergence.
On the other hand, because the value is continuous (Lemma 2 in \citet{mertens_repeated_1986}
or Proposition III.4.3. in \citet{mertens_sorin_zamir_2015}), the
two topologies are equivalent. 
\end{proof}
Theorem \ref{thm5} suggests a possible alternative construction of
the set $\Pi$ of consistent probability over the universal belief
space. The alternative construction is simply based on the values
of finite zero-sum Bayesian games.

\subsection{Strategic topologies}

Following \citet{dekel_topologies_2006}, the literature studies the
robustness of the behavior with respect to beliefs in a systematic
way. A typical approach is to define the $\varepsilon$-ICR as the
outcome of the iterated elimination procedure, where, at each stage
of the elimination, players actions must be $\varepsilon$-best responses
to the conjectures that allow for arbitrary correlations with the
state of the world and that assign probability 1 to actions that survived
the previous stage. This solution concept depends only on the hierarchy
of beliefs. One defines two topologies (see \citet{chen_uniform_2016}): 
\begin{itemize}
\item sequence of types (i.e., hierarchies of beliefs) $t_{n}$ converges
to $t$ in the \emph{weak strategic topology} if and only if \uwave{for
each game }%
\mbox{%
$G$%
}, the set of actions that are 0-ICR at type $t$ in game $G$ is equal
to the set of actions such that \uwave{for each }%
\mbox{%
$\varepsilon>0$%
}\uwave{, there exists }%
\mbox{%
$N$%
} such that for each $n\geq N$, the action is $\varepsilon-ICR$ at
type $t_{n}$ in game $G$, and 
\item sequence of types (i.e., hierarchies of beliefs) $t_{n}$ converges
to $t$ in the \emph{uniform strategic topology} if and only if \uwave{for
each }%
\mbox{%
$\varepsilon>0$%
}\uwave{, there exists }%
\mbox{%
$N$%
}such that \uwave{for each game }%
\mbox{%
$G$%
} (with payoffs bounded in norm by 1), the set of actions that are
0-ICR at type $t$ in game $G$ is equal to the set of actions that
are $\varepsilon-ICR$ at type $t_{n}$ for $n\geq N$. 
\end{itemize}
The two topologies differ by the order of quantifiers; the uniform
topology corresponds more to the value-based distance topology and
the weak topology corresponds to the topology introduced in Section
\ref{subsec:Weak-zero-sum-topology}. Both topologies are metric.
\citet{chen_uniform_2016} provide a full characterization of the
two topologies directly in terms of hierarchies of beliefs. In particular,
it is well-known that the two topologies are strictly finer than the
product topology on the universal type space.

It is not entirely obvious how to compare the value-based and the
strategic topologies as, due to the difference between ex ante and
interim approaches, their domains are different.

One way to do is to equip $\Theta^{1}$ with the strategic or uniform
strategic topology, and to consider the induced weak topology (i.e.
of convergence in distribution) on $\Pi\subset\Delta(\Theta^{1})$.

One can easily see that the two topologies are strictly finer than
the weak topology on $\Pi_{0}$ (this is demonstrated by, for instance
the Rubinstein's email game - see Example \ref{ex:Rubinstein}). Given
Theorem \ref{Thm: weak topology}, the two strategic topologies are
strictly finer than the value-based topology on $\Pi_{0}$. We do
not know whether the two topologies remain finer when the value-based
topology is extended to $\Pi$.

\section{Conclusion}

In this paper, we have introduced and analyzed the value-based distance
on the space of information structures. The main advantage of the
definition is that it has a simple and useful interpretation as the
tight upper bound on the loss or gain from moving between two information
structures. This allows us to directly apply it to numerous questions
about the value of information, the relation between the games and
single-agent problems, comparison of information structures, etc.
Additionally, we show that the distance contains an interesting topological
information. On one hand, the topology induced on the countable information
structures is equivalent to the topology of weak convergence of consistent
probabilities over coherent hierarchies of beliefs. On the other hand,
the set of countable information structures is not totally bounded
for the value-based distance, which solves negatively the last open
question raised in \citet{mertens_repeated_1986}, with deep implications
for stochastic games.

By restricting our attention to zero-sum games, we were able to re-examine
the relevance of many phenomena observed and discussed in the strategic
discontinuities literature. On one hand, the distinction between the
approximate knowledge and the approximate common knowledge is not
important in situations of conflict. On the other hand, the higher
order beliefs matter on some, potentially uncountably large structures.
More generally, we believe that the discussion of the strategic phenomena
on particular classes of games can be fruitful line of future research.
It is not the case that each problem must involve coordination games.
Interesting classes of games to study could be common interest games,
potential games, etc. \footnote{As an example of work in this direction, \citet{yamashita2018} studies an order on hierarchies and types induced by payoffs in supermodular games.}

\appendix
%dummy comment inserted by tex2lyx to ensure that this paragraph is not empty

\section{Proofs of Section \ref{sec:Computing}}

\subsection{Proof of Theorem \ref{thm1}}

The proof of Theorem 1 relies on two main aspects: the two interpretations
of a garbling (deterioration of signals, and strategy), and the minmax
theorem.

\emph{Part 1.} We start with general considerations. Recall that any
$g\in\CG(L)$ is (identified with) a map from $K\times\N\times\N\rightarrow[-1,1]$
such that $g(k,i,j)=-1$ if $i\geq L$ and $g(k,i,j)=1$ if $i<L$
and $j\geq L$. For $u\in\CU$ and $g\in\CG(L)$, we denote by $\gamma_{u,g}(q_{1},q_{2})$
the payoff of player 1 in the zero-sum game $\Gamma(u,g)$ when player
1 plays $q_{1}\in\CQ$ and player 2 plays $q_{2}\in\CQ$. Extending
as usual $g$ to mixed actions, we have: $\gamma_{u,g}(q_{1},q_{2})=\sum_{k,c,d}u(k,c,d)g(k,q_{1}(c),q_{2}(d))$.
For $u\in\CU$ and $g\in\CG$, the scalar product $\langle g,u\rangle=\sum_{k,c,d}g(k,c,d)u(k,c,d)$
is well defined, and corresponds to the expectation of $g$ with respect
to $u$, and to the payoff $\gamma_{u,g}(Id,Id)$, where $Id\in\CQ$
is the strategy which plays with probability one the signal received.
The map $g\in\CG(L)$ has been extended to $K\times\N\times\N$ in
such a way that 
\[
\forall q_{1}\in\CQ(L),\;\min_{q_{2}\in\CQ(L)}\gamma_{u,g}(q_{1},q_{2})=\min_{q_{2}\in\CQ}\gamma_{u,g}(q_{1},q_{2}),
\]
and a similar result for player $1$. Therefore, we have for all $g\in\CG(L)$
\begin{align*}
\val(u,g) & =\max_{\sigma_{1}\in\CQ(L)}\min_{\sigma_{2}\in\CQ(L)}\gamma_{u,g}(\sigma_{1},\sigma_{2})=\min_{\sigma_{2}\in\CQ(L)}\max_{\sigma_{1}\in\CQ(L)}\gamma_{u,g}(\sigma_{1},\sigma_{2})\\
 & =\max_{\sigma_{1}\in\CQ}\min_{\sigma_{2}\in\CQ}\gamma_{u,g}(\sigma_{1},\sigma_{2})=\min_{\sigma_{2}\in\CQ}\max_{\sigma_{1}\in\CQ}\gamma_{u,g}(\sigma_{1},\sigma_{2})
\end{align*}

Let us now compute the payoff $\gamma_{u,g}(q_{1},q_{2})$, for any
$q_{1},q_{2}$: 
\begin{eqnarray*}
\gamma_{u,g}(q_{1},q_{2}) & = & \sum_{k,c,d}u(k,c,d)g(k,q_{1}(c),q_{2}(d))\\
 & = & \sum_{k,c,d}u(k,c,d)\sum_{c',d')}q_{1}(c)(c')q_{2}(d)(d')g(k,c',d')\\
 & = & \sum_{k,c',d'}g(k,c',d')\sum_{c,d}u(k,c,d)q_{1}(c)(c')q_{2}(d)(d')\\
 & = & \sum_{k,c',d'}g(k,c',d')\left(\left(q_{1}.u.q_{2}\right)(k,c',d')\right)\\
 & = & \langle g,q_{1}.u.q_{2}\rangle.
\end{eqnarray*}
Consequently, $\val(u,g)=\max_{q_{1}\in\CQ}\min_{q_{2}\in\CQ}\langle g,q_{1}.u.q_{2}\rangle=\min_{q_{2}\in\CQ}\max_{q_{1}\in\CQ}\langle g,q_{1}.u.q_{2}\rangle$.
Since both players can play the $Id$ strategy in $\Gamma({u,g})$,
we obtain for all $u\in\CU$ and $g\in\CG(L)$ : 
\[
\inf_{q_{2}\in\CQ}\langle g,u.q_{2}\rangle\leq\val(u,g)\leq\sup_{q_{1}\in\CQ}\langle g,q_{1}.u\rangle.
\]
Notice also that for all $u$, $v$ in $\CU$, $\|u-v\|=\sup_{g\in\CG}\langle g,u-v\rangle$.

\noindent \emph{Part 2.} We now prove Theorem \ref{thm1}. Take $g\in\CG$
and $q_{1},q_{2}\in\CQ$. Then, $\val(v.q_{2},g)\geq\val(v,g)\text{ and }\val(u,g)\geq\val(q_{1}.u,g)$,
which implies 
\[
\val(v,g)-\val(u,g)\leq\val(v.q_{2},g)-\val(q_{1}.u,g)\leq\|q_{1}.u-v.q_{2}\|.
\]
Because $g$ and $q_{1},q_{2}$ are arbitrary, we obtain 
\[
\sup_{g\in\CG}\left(\val(v,g)-\val(u,g)\right)\leq\inf_{q_{1}\in\CQ,q_{2}\in\CQ}\|q_{1}.u-v.q_{2}\|.
\]
Thus, it remains to prove that 
\begin{equation}
\inf_{q_{1}\in\CQ,q_{2}\in\CQ}\|q_{1}.u-v.q_{2}\|\leq\sup_{g\in\CG}\left(\val(v,g)-\val(u,g)\right)\label{eq1}
\end{equation}
and that the minimum is achieved in the left-hand side. We have 
\[
\inf_{q_{1}\in\CQ,q_{2}\in\CQ}\|q_{1}.u-v.q_{2}\|=\inf_{q_{1}\in\CQ,q_{2}\in\CQ}\sup_{g\in\CG}\langle g,v.q_{2}-q_{1}.u\rangle.
\]
The sets $\CQ$ are convex and compact in the product topology, the
set $\CG$ is a convex subset of some linear vector space, the map
\[
(g,(q_{1},q_{2}))\rightarrow\langle g,v.q_{2}-q_{1}.u\rangle,
\]
is bilinear and continuous with respect to $(q_{1},q_{2})\in\CQ^{2}$
for any fixed $g\in\CG$, and by Sion's theorem (see e.g. Sorin \citet{sorin2002first}
Proposition A.8): 
\[
\min_{q_{1}\in\CQ,q_{2}\in\CQ}\sup_{g\in\CG}\langle g,v.q_{2}-q_{1}.u\rangle=\sup_{g\in\CG}\min_{q_{1}\in\CQ,q_{2}\in\CQ}\langle g,v.q_{2}-q_{1}.u\rangle.
\]
Inequality (\ref{eq1}) now follows from : 
\begin{eqnarray*}
\sup_{g\in\CG}\inf_{q_{1}\in\CQ,q_{2}\in\CQ}\langle g,v.q_{2}-q_{1}.u\rangle & = & \sup_{g\in\CG}\left(\inf_{q_{2}\in\CQ}\langle g,v.q_{2}\rangle-\sup_{q_{1}\in\CQ}\langle g,q_{1}.u\rangle\right)\\
 & \leq & \sup_{g\in\CG}\left(\val(v,g)-\val(u,g)\right).
\end{eqnarray*}
Finally, if $u,v\in\CG(L)$ for some $L\in\N$, the argument of part
$2$ can be repeated by replacing $\CQ$ by $\CQ(L)$ and $\CG$ by
$\CG(L)$ since for any $q_{1},q_{2}\in\CQ(L)$, we have $\|q_{1}.u-v.q_{2}\|=\sup_{g\in\CG(L)}\langle g,q_{1}.u-v.q_{2}\rangle$.
Moreover, $\CG(L)$ is compact for the topology of pointwise convergence
on $K\times\{0,...,L-1\}^{2}$, and the map $(g,(q_{1},q_{2}))\rightarrow\langle g,v.q_{2}-q_{1}.u\rangle$
is continuous with respect to $g\in\CG(L)$ for any fixed $(q_{1},q_{2})\in\CQ(L)$,
implying that 
\begin{equation}
\max_{g\in\CG(L)}\left(\val(v,g)-\val(u,g)\right)=\min_{q_{1}\in\CQ(L),q_{2}\in\CQ(L)}\|q_{1}.u-v.q_{2}\|.\label{eq:proofthm1_1}
\end{equation}
Clearly, 
\begin{align*}
\sup_{g\in\CG(L)}\left(\val(v,g)-\val(u,g)\right) & \leq\sup_{g\in\CG}\left(\val(v,g)-\val(u,g)\right),\text{ and}\\
\inf_{q_{1}\in\CQ,q_{2}\in\CQ}\|q_{1}.u-v.q_{2}\| & \leq\inf_{q_{1}\in\CQ(L),q_{2}\in\CQ(L)}\|q_{1}.u-v.q_{2}\|.
\end{align*}
and using \eqref{eq:d characterization} and \eqref{eq:proofthm1_1},
these inequalities are equalities, which concludes the proof.

\subsection{Proof of Proposition \ref{prop: SA=00003D00003DZS}}

Let us start with general properties of $d_{1}$. Let us define the
set of single agent information structures as $\CU_{1}=\Delta(K\times\N)$
using the same convention that countable sets are identified with
subsets of $\N$. Note that given $u\in\Delta(K\times C\times D)$,
$\marg_{K\times C}u\in\CU_{1}$. Let 
\[
\CG'_{1}=\{g':K\times\N\to[-1,1],\;\exists L\;s.t.\forall i\geq L,\;g'(k,i)=-1\}
\]
be the set of single-agent decision problems, and define for $u',v'\in\CU_{1}$,
$d'_{1}(u',v')=\sup_{g'\in\CG'_{1}}|\val(v',g')-\val(u',g')|$. It
is easily seen that for any $u,v\in\Delta(K\times C\times D)$, 
\begin{equation}
d_{1}(u,v)=d'_{1}(u',v')=\max\{\min_{q\in\CQ}\|u'-q.v'\|,\min_{q\in\CQ}\|q.u'-v'\|\}\label{eq:d_1_marg}
\end{equation}
where $u'=\marg_{K\times C}u$, $v'=\marg_{K\times C}v$, $q.u'(k,c)=\sum_{s\in C}u'(k,s)q(s)(c)$
and where the last equality can be obtained by mimicking (and simplifying)
the arguments of the proof of Theorem \ref{thm1}.

We now prove Theorem \ref{thm1}. Using the assumptions, we have $u(k)=v(k)$,
$u\left(c,d|k\right)=u\left(c|k\right)u\left(d|k\right)$ and $v\left(c^{\prime},d|k\right)=v\left(d|k\right)v\left(c^{\prime}|k\right)=u(d|k)v\left(c^{\prime}|k\right)$.
For any pair of garblings $q_{1},q_{2}$ 
\begin{align*}
\left\Vert u.q_{2}-q_{1}.v\right\Vert = & \sum_{k,c,d}\left|\sum_{\beta}u\left(k,c,\beta\right)q_{2}\left(d|\beta\right)-\sum_{\alpha}v\left(k,\alpha,d\right)q_{1}\left(c|\alpha\right)\right|\\
= & \sum_{k,c}u\left(k\right)\sum_{d}\left|u\left(c|k\right)\sum_{\beta}u\left(\beta|k\right)q_{2}\left(d|\beta\right)-\left(\sum_{\alpha}v\left(\alpha|k\right)q_{1}\left(c|\alpha\right)\right)u\left(d|k\right)\right|\\
= & \sum_{k,c}u\left(k\right)\sum_{d}\left|u\left(d|k\right)\Gamma\left(k,c\right)+\Delta\left(k,d\right)u\left(c|k\right)\right|,
\end{align*}
where 
\begin{align*}
\Delta\left(k,d\right) & =u\left(d|k\right)-\sum_{\beta}u\left(\beta|k\right)q_{2}\left(d|\beta\right),\\
\Gamma\left(k,c\right) & =\sum_{\alpha}v\left(\alpha|k\right)q_{1}\left(c|\alpha\right)-u\left(c|k\right).
\end{align*}
Because $\left|x+y\right|\geq|x|+\text{sgn}(x)y$ for each $x,y\in\R$,
we have 
\begin{align*}
 & \sum_{d}\left|u\left(d|k\right)\Gamma\left(k,c\right)+\Delta\left(k,d\right)u\left(c|k\right)\right|\\
\geq & \sum_{d}u\left(d|k\right)\left|\Gamma\left(k,c\right)\right|+\text{sgn}\left(\Gamma\left(k,c\right)\right)u\left(c|k\right)\sum_{d}\Delta\left(k,d\right)\\
= & \sum_{d}u\left(d|k\right)\left|\Gamma\left(k,c\right)\right|.
\end{align*}
where the last equality comes from the fact that $\sum_{d}\Delta\left(k,d\right)=0$.
Thus, we obtain 
\begin{align*}
\left\Vert u.q_{2}-q_{1}.v\right\Vert  & \geq\sum_{k,c,d}u\left(k\right)\left|u\left(d|k\right)\Gamma\left(k,c\right)\right|\\
 & =\sum_{k,c,d}u\left(k\right)\left|u\left(d|k\right)u\left(c|k\right)-\sum_{\alpha}u\left(d|k\right)v\left(\alpha|k\right)q_{1}\left(c|\alpha\right)\right|\\
 & =\left\Vert u-q_{1}.v\right\Vert .
\end{align*}

We deduce that 
\begin{align*}
\min_{q_{1},q_{2}}\left\Vert u.q_{2}-q_{1}.v\right\Vert =\min_{q_{1}}\left\Vert u-q_{1}.v\right\Vert .
\end{align*}
Inverting the roles of the players, we also have 
\begin{align*}
\min_{q_{1},q_{2}}\left\Vert v.q_{2}-q_{1}.y\right\Vert =\min_{q_{1}}\left\Vert v-q_{1}.u\right\Vert .
\end{align*}
We conclude that 
\begin{align*}
d(u,v) & =\max\{\min_{q_{1},q_{2}}\left\Vert u.q_{2}-q_{1}.v\right\Vert ;\min_{q_{1},q_{2}}\left\Vert v.q_{2}-q_{1}.y\right\Vert \}\\
 & =\max\{\min_{q_{1}}\left\Vert u-q_{1}.v\right\Vert ;\min_{q_{1}}\left\Vert v-q_{1}.u\right\Vert \}\\
 & =d_{1}(u,v),
\end{align*}
where the last equality follows from \eqref{eq:d_1_marg} together
with the fact that $\marg_{K\times D}u=\marg_{K\times D}v$.

\subsection{Proof of Corollary \ref{cor: d charact comparison}}

By Theorem \ref{thm1}, 
\[
\sup_{g\in\CG}\left(\val(v,g)-\val(u,g)\right)=\min_{q_{1}\in\CQ,q_{2}\in\CQ}\|q_{1}.u-v.q_{2}\|.
\]
Since $q_{1}.u\preceq u$ and $v.q_{2}\succeq v$, we have $\sup_{g\in\CG}\left(\val(v,g)-\val(u,g)\right)\geq\inf_{u'\preceq u,v'\succeq v}\|u'-v'\|$.
Consider now $u'\preceq u$ and $v'\succeq v$. We have for all $g\in\CG$:
\[
\|u'-v'\|\geq\val(v',g)-\val(u',g)\geq\val(v,g)-\val(u,g),
\]
so $\sup_{g\in\CG}\left(\val(v,g)-\val(u,g)\right)\leq\inf_{u'\preceq u,v'\succeq v}\|u'-v'\|$.
Hence 
\[
\sup_{g\in\CG}\left(\val(v,g)-\val(u,g)\right)=\inf_{u'\preceq u,v'\succeq v}\|u'-v'\|
\]

\subsection{Proof of Proposition \ref{prop:diameter}}

We prove the lower bound of (\ref{eq:diameter bounds}). Let $g\left(k\right)=\1_{p_{k}>q_{k}}-\1_{p_{k}\leq q_{k}}$.
Then, 
\[
d\left(u,v\right)\geq\val\left(u,g\right)-\val\left(v,g\right)=\sum_{k\in K}\left(p_{k}-q_{k}\right)g\left(k\right)=\sum_{k\in K}\left|p_{k}-q_{k}\right|.
\]

Let us prove the upper bound of (\ref{eq:diameter bounds}). Define
$\bar{u}$ and $\underline{v}$ in $\Delta(K\times K_{C}\times K_{D})$
with $K=K_{C}=K_{D}$ such that $\bar{u}(k,c,d)=p_{k}\1_{c=k}\1_{d=k_{0}}$
for some fixed $k_{0}\in K$ (complete information for player 1, trivial
information for player 2, and the same prior about $k$ as $u$) and
$\underline{v}(k,c,d)=q_{k}\1_{c=k_{0}}\1_{d=k}$ for all $(k,c,d)$
(trivial information for player 1, complete information for player
2, and the same beliefs about $k$ as $v$). Since the value of a
zero-sum game is weakly increasing with player 1's information and
weakly decreasing with player 2's information, we have $u\preceq\bar{u}$
and $\underline{v}\preceq v$. Therefore, 
\[
\sup_{g\in\CG}(\val(u,g)-\val(v,g))\leq\sup_{g\in\CG}(\val(\bar{u},g)-\val(\underline{v},g))=\min_{q_{1}\in\CQ,q_{2}\in\CQ}\|\bar{u}.q_{2}-q_{1}.\underline{v}\|,
\]
where, according to Theorem \ref{thm1}, the minimum in the last expression
is attained for garblings with values in $\Delta K$. Since player
2 has a unique signal in $\bar{u}$, only $q_{2}(.|k_{0})\in\Delta K$
matters. We denote it by $q'=q_{2}\left(.|k_{0}\right)$. Similarly,
we define $p^{\prime}=q_{1}(.|k_{0})\in\Delta(K)$. Then,

\begin{align*}
\|\bar{u}.q_{2}-q_{1}.\underline{v}\| & =\sum_{(k,c,d)\in K^{3}}|p_{k}\1_{c=k}q{}_{d}^{\prime}-q_{k}\1_{d=k}p{}_{c}^{\prime}|\\
 & =\sum_{k\in K}|p_{k}q{}_{k}^{\prime}-q_{k}p{}_{k}^{\prime}|+p_{k}(1-q{}_{k}^{\prime})+q_{k}(1-p{}_{k}^{\prime})\\
 & =2+\sum_{k\in K}|p_{k}q_{k}^{\prime}-q_{k}p_{k}^{\prime}|-p_{k}q_{k}^{\prime}-q_{k}p_{k}^{\prime}\\
 & =2\left(1-\sum_{k\in K}\min\left(p_{k}q_{k}^{\prime},q_{k}p_{k}^{\prime}\right)\right).
\end{align*}
A similar inequality holds by inverting the roles of $u$ and $v$,
and the upper bound follows from the fact that one can choose arbitrary
$p^{\prime},q^{\prime}$.

If $p=q$, then 
\[
\sum_{k\in K}\min\left(p_{k}q_{k}^{\prime},q_{k}p_{k}^{\prime}\right)=\sum_{k\in K}p_{k}\min\left(q_{k}^{\prime},p_{k}^{\prime}\right)\leq\sum_{k\in K}p_{k}p_{k}^{\prime}\leq\max_{k\in K}p_{k},
\]
where the latter is attained by $p_{k}^{\prime}=q_{k}^{\prime}=\1_{\{k=k^{*}\}}$
for some $k^{*}\in K$ such that $p_{k^{*}}=\max_{k\in k}p_{k}$.

\subsection{Proof of Proposition \ref{prop: optimal strategies}}

We prove the first claim. Consider any strategy $\tau$ of player
2 in $\CQ$. Notice that $(\sigma.q_{1}).u=\sigma.(q_{1}.u)$ and
$(v.q_{2}).\tau=v.(\tau.q_{2}).$ Using the notation of the proof
of theorem \ref{thm1}, we have 
\begin{eqnarray*}
\gamma_{u,g}(\sigma.q_{1},\tau) & = & \langle g,(\sigma.q_{1}).u.\tau\rangle\\
 & = & \langle g,\sigma.(q_{1}.u).\tau\rangle\\
 & \geq & \langle g,\sigma.(v.q_{2}).\tau\rangle-\|q_{1}.u-v.q_{2}\|\\
 & \geq & \langle g,\sigma.v.(\tau.q_{2})\rangle-d(u,v)\\
\  & \geq & \val(v,g)-d(u,v)\\
 & \geq & \val(u,g)-2d(u,v),
\end{eqnarray*}
so $\sigma.q_{1}$ is $2d(u,v)$-optimal in $g$ on $u$.

\section{Proofs of Section \ref{sec:Value-of-additional}}

\subsection{Proof of Proposition \ref{prop. Info substitutes}}

Because $u$ has more information than $v$, 
\begin{align*}
d\left(u,v\right) & =\min_{q_{2}\in\CQ}\min_{q_{1}\in\CQ}\left\Vert u.q_{2}-q_{1}.v\right\Vert \leq\min_{q_{2}\in\CQ}\min_{q_{1}:C\rightarrow\Delta\left(C\times C_{1}\times C_{2}\right)}\left\Vert u.q_{2}-q_{1}.v\right\Vert ,
\end{align*}
where in the second inequality we use a restricted set of player 1's
garblings that do not depend on $c_{1}$. Further, for any such $q_{1}$
and an arbitrary garbling $q_{2}$, we have 
\begin{align*}
\left\Vert u.q_{2}-q_{1}.v\right\Vert  & =\sum_{k,c,c_{1},c_{2},d}\left|\sum_{\beta}u\left(k,c,c_{1},c_{2},\beta\right)q_{2}\left(d|\beta\right)-\sum_{\alpha}u\left(k,c,c_{1},d\right)q_{1}\left(c,c_{2}|\alpha\right)\right|\\
 & =\sum_{k,c,c_{1},c_{2},d}u\left(k,c_{1}\right)\left|\sum_{\beta}u\left(c,c_{2},\beta|k,c_{1}\right)q_{2}\left(d|\beta\right)-\sum_{\alpha}u\left(c,d|k,c_{1}\right)q_{1}\left(c,c_{2}|\alpha\right)\right|.
\end{align*}
Because of the conditional independence assumption, the above is equal
to 
\begin{align*}
= & \sum_{k,c,c_{2},d}\left(\sum_{c_{1}}u\left(k,c_{1}\right)\right)\left|\sum_{\beta}u\left(c,c_{2},\beta|k\right)q_{2}\left(d|\beta\right)-\sum_{\alpha}u\left(c,d|k\right)q_{1}\left(c,c_{2}|\alpha\right)\right|\\
= & \sum_{k,c,c_{2},d}\left|\sum_{\beta}u\left(k,c,c_{2},\beta\right)q_{2}\left(d|\beta\right)-\sum_{\alpha}u\left(k,c,d\right)q_{1}\left(c,c_{2}|\alpha\right)\right|\\
= & \left\Vert u^{\prime}.q_{2}-q_{1}.v{}^{\prime}\right\Vert .
\end{align*}
Hence, 
\[
d\left(u,v\right)\leq\min_{q_{2}}\min_{q_{1}:C\rightarrow\Delta\left(C\times C_{1}\times C_{2}\right)}\left\Vert u^{\prime}.q_{2}-q_{1}.v{}^{\prime}\right\Vert =d\left(u^{\prime},v{}^{\prime}\right).
\]

\subsection{Proof of Proposition \ref{prop. Info complements}}

We have 
\begin{align*}
d\left(u^{\prime},v{}^{\prime}\right)=d_{1}\left(u^{\prime},v{}^{\prime}\right)=d_{1}\left(u,v\right)\leq d\left(u,v\right).
\end{align*}
The first equality comes from Proposition \ref{prop: Games vs Single agent},
and the second from the fact that $u$ and $u'$ (resp. $v$ and $v'$)
induce the same distribution on player $1$ first order beliefs, and
the inequality from the definition of the two distances.

\subsection{Proof of Proposition \ref{prop: joint info}}

We have the following Lemma: 
\begin{lem}
\label{lem: cond independence}Suppose that $c_{1}$ is $\varepsilon$-conditionally
independent from $\left(k,d\right)$ given $c$. Then, 
\[
\sup_{g\in\CG}\val\left(u,g\right)-\val\left(v,g\right)\leq\varepsilon.
\]
\end{lem}

\begin{proof}
Let $q_{2}:D\times D_{1}\rightarrow D$ be defined as $q_{2}\left(d,d_{1}\right)(d')=\1_{d'=d}$.
Let $q_{1}:C\rightarrow C\times C_{1}$ be defined as $q_{1}\left(c,c_{1}|c\right)=u\left(c_{1}|c\right)$.
Then, 
\begin{align*}
\left\Vert u.q_{2}-q_{1}.v\right\Vert  & =\sum_{k,c,c_{1},d}\left|u\left(k,c,c_{1},d\right)-u\left(k,c,d\right)u\left(c_{1}|c\right)\right|\\
 & =\sum_{c}u\left(c\right)\sum_{k,c_{1},d}\left|u\left(k,c_{1},d|c\right)-u\left(k,d|c\right)u\left(c_{1}|c\right)\right|\leq\varepsilon.
\end{align*}
The claim follows from Theorem \ref{thm1}. 
\end{proof}
An analogous argument shows that, if $d_{1}$ is $\varepsilon$-conditionally
independent from $\left(k,c\right)$ given $d$, then 
\[
\sup_{g\in\CG}\val\left(v,g\right)-\val\left(u,g\right)\leq\varepsilon.
\]

\section{Proofs of Section \ref{sec:Topology}}

\subsection{Theorem \ref{Thm: weak topology}: the weak topology is contained
in the value-based topology}

Because the weak topology is metrisable, it is enough to show that
if $u_{n}\in\Delta\left(K\times C_{n}\times D_{n}\right)$ and $u\in\Delta\left(K\times C\times D\right)$
are information structures such that $d\left(u_{n},u\right)\rightarrow0$,
then $\tilde{u}_{n}$ converges to $\tilde{u}$ in the weak topology.
Let $\tilde{u}^{*}$ be an accumulation point of sequence $\tilde{u}_{n}$.
Such a point exists because $\Pi$ is compact.

Suppose that $\tilde{u}^{*}\neq\tilde{u}$. Then, by lemma 41 in Gossner
Mertens \citet{gossner_value_2001}), there is a game $g$ such that
$\left|\val\left(\tilde{u},g\right)-\val\left(\tilde{u}^{*},g\right)\right|>0.$
For a fixed game, $\val:\Pi\rightarrow\left[-1,1\right]$ is continuous
in the weak topology (see Lemma 2 in \citet{mertens_repeated_1986}
or Proposition III.4.3. in \citet{mertens_sorin_zamir_2015}). It
follows that $\lim_{n}\left|\val\left(\tilde{u},g\right)-\val\left(\tilde{u}_{n},g\right)\right|=\left|\val\left(\tilde{u},g\right)-\val\left(\tilde{u}^{*},g\right)\right|>0$.
But this contradicts $d\left(u_{n},u\right)=d\left(\tilde{u}_{n},\tilde{u}\right)\rightarrow0$.

\subsection{Theorem \ref{Thm: weak topology}: the value-based topology is contained
in the weak topology}

Because the weak topology is metrisable, it is enough to show that
if $u_{n}\in\Delta\left(K\times C_{n}\times D_{n}\right)$ and $u\in\Delta\left(K\times C\times D\right)$
are information structures such that $\tilde{u}_{n}$ converges to
$\tilde{u}$ in the weak topology, then $d\left(u_{n},u\right)\rightarrow0$.
We will show that for each $\varepsilon>0$, 
\begin{equation}
\limsup_{n\rightarrow\infty}\sup_{g\in\CG}\left(\val\left(u_{n},g\right)-\val\left(u,g\right)\right)\leq0.\label{eq:topologies 1}
\end{equation}
Because we can switch the roles of players, this will suffice to establish
our claim.

\emph{Partitions of unity}. We can without loss of generality assume
that $u$ is non-redundant and all signals $c$ and $d$ have positive
probability. We can associate signals $c\in C\subseteq\N$ and $d\in D\subseteq\N$
with the corresponding hierarchies of beliefs in $\Theta_{1}$ and
$\Theta_{2}$. In other words, we identify $C\subseteq\Theta_{1}$
as the (countable) support of $\tilde{u}$ and $D\subseteq\Theta_{2}$
as the smallest countable set such that for each $c\in C$, $\phi_{1}\left(K\times D|c\right)=1$
(i.e., $D$ is the union of countable supports of all beliefs of hierarchies
in $C$). For each $c\in C$ and $d\in D$, we denote the corresponding
hierarchies under $u$ as $\tilde{c}$ and $\tilde{d}$. Also, let
$C^{m}=C\cap\left\{ 0,...,m\right\} $ and $D^{m}=D\cap\left\{ 0,...,m\right\} $.

Because $\Theta_{2}$ is Polish, for each $m\in\N$ and each $d\in D^{m}$,
we can find continuous functions $\kappa_{d}^{m}:\Theta_{2}\rightarrow\left[0,1\right]$
for $m\in\N,d\in\{0,...,m\}$ such that 
\begin{align*}
\kappa_{d}^{m}\left(\tilde{d}\right) & =1\text{ for each }d\in D^{m},\\
\kappa_{d}^{m} & \equiv0\text{ if }d\notin D,\text{ and }\\
\sum_{d=0}^{m}\kappa_{d}^{m}\left(\theta_{2}\right) & =1\text{ for each }\theta_{2}\in\Theta_{2}.
\end{align*}
In other words, for each $m$, $\left\{ \kappa_{d}^{m}\right\} _{0\leq d\leq m}$
is a continuous partition of unity on space $\Theta_{2}$ with the
property that for each $d\in D^{m}$, $\kappa_{d}^{m}$ peaks at hierarchy
$\tilde{d}$.

Notice that for each $c\in C$, each $d\in D^{p}$, we have 
\[
\E_{\phi_{1}(\tilde{c})}[\1_{\{k\}}(.)\kappa_{d}^{p}(.)]\geq u\left(k,d|c\right),
\]
and 
\[
\sum_{k\in K}\sum_{d=0}^{p}\left|\E_{\phi_{1}(\tilde{c})}[\1_{\{k\}}(.)\kappa_{d}^{p}(.)]-u\left(k,d|c\right)\right|=u(D\setminus D^{p}|c).
\]
Because all hierarchies $\tilde{c},c\in C$ are distinct, for each
$m$, there exists $p^{m}<\infty$ and $\varepsilon^{m}\in\left(0,\frac{1}{m}\right)$
such that for any $c,c'\in C^{m}$ such that $c\neq c{}^{\prime}$,
\[
\sum_{k\in K}\sum_{d=0}^{p^{m}}\left|\E_{\phi_{1}(\tilde{c})}[\1_{\{k\}}\kappa_{d}^{p^{m}}]-\E_{\phi_{1}(\tilde{c}^{\prime})}[\1_{\{k\}}\kappa_{d}^{p^{m}}]\right|\geq2\varepsilon^{m}.
\]
Let 
\[
h_{c}^{m}\left(\theta_{1}\right)=\sum_{k}\sum_{d=0}^{p^{m}}\left|\E_{\phi_{1}(\theta_{1})}[\1_{\{k\}}\kappa_{d}^{p^{m}}]-\E_{\phi_{1}(\tilde{c})}[\1_{\{k\}}\kappa_{d}^{p^{m}}]\right|.
\]
Then, $h_{c}^{m}$ is a continuous function such that $h_{c}^{m}\left(\tilde{c}\right)=0$,
and such that if $h_{c}^{m}\left(\theta_{1}\right)\leq\varepsilon^{m}$
for some $c\in C^{m}$, then $h_{c^{\prime}}^{m}\left(\theta_{1}\right)\geq\varepsilon^{m}$
for any $c'\in C^{m}$ such that $c^{\prime}\neq c$. Define continuous
functions for $0\leq c\leq m+1$, 
\begin{align*}
\kappa_{c}^{m}\left(\theta_{1}\right) & =\max\left(1-\frac{1}{\varepsilon^{m}}h_{c}^{m}\left(\theta_{1}\right),0\right)\text{ for }c\in C_{m},\\
\kappa_{c}^{m} & \equiv0\text{ if }c\notin C,\text{ and }\\
\kappa_{m+1}^{m}\left(\theta_{1}\right) & =1-\sum_{c=0}^{m}\kappa_{c}^{m}\left(\theta_{1}\right).
\end{align*}
Then, for each $m$, $\sum_{c=0}^{m+1}\kappa_{c}^{m}\equiv1$, and
$\kappa_{c}^{m}\left(\theta_{1}\right)\in\left[0,1\right]$ for each
$c=0,...,m+1$, which implies that $\left\{ \kappa_{c}^{m}\right\} _{0\leq c\leq m+1}$
is a continuous partition of unity on space $\Theta_{1}$ such that
for each $c\in C^{m}$, $\kappa_{c}^{m}(\tilde{c})=1$.

\emph{Conditional independence}. For each information structure $v\in\Delta\left(K\times C{}^{\prime}\times D{}^{\prime}\right)$,
define an information structure 
\[
K^{m}v\in\Delta\left(K\times C{}^{\prime}\times\left\{ 0,...,m+1\right\} \times D{}^{\prime}\times\left\{ 0,...,p^{m}\right\} \right)
\]
so that 
\[
K^{m}v\left(k,c^{\prime},\hat{c},d^{\prime},\hat{d}\right)=v\left(k,c^{\prime},d^{\prime}\right)\kappa_{\hat{c}}^{m}\left(\tilde{c}^{\prime}\right)\kappa_{\hat{d}}^{p^{m}}\left(\tilde{d}^{\prime}\right).
\]
Let $\delta^{m}v=2\varepsilon^{m}+K^{m}v\left(\hat{c}=m+1\right).$

We are going to show that, under $K^{m}v$, signal $c^{\prime}$ is
$\delta^{m}v$-conditionally independent from $\left(k,\hat{d}\right)$
given $\hat{c}$. Notice first that, if $K^{m}v\left(k,d^{\prime},\hat{d},c{}^{\prime},\hat{c}\right)>0$
for some $\hat{c}\in C^{m}$, then $h_{\hat{c}}^{m}\left(\tilde{c}^{\prime}\right)\leq\varepsilon^{m}.$
It follows that 
\begin{align*}
 & \sum_{k}\sum_{\hat{d}=0}^{p^{m}}\left|K^{m}v\left(k,\hat{d}|\hat{c},c{}^{\prime}\right)-\E_{\phi_{1}\left(\tilde{\hat{c}}\right)}[\1_{\{k\}}\kappa_{\hat{d}}^{p^{*}}]\right|\\
 & =\sum_{k}\sum_{\hat{d}=0}^{p^{m}}\left|K^{m}v\left(k,\hat{d}|c{}^{\prime}\right)-\E_{\phi_{1}\left(\tilde{\hat{c}}\right)}[\1_{\{k\}}[\kappa_{\hat{d}}^{p^{m}}]\right|\\
 & =\sum_{k}\sum_{\hat{d}=0}^{p^{m}}\left|\E_{\phi_{1}(\tilde{c}^{\prime})}[\1_{\{k\}}\kappa_{\hat{d}}^{p^{m}}]-\E_{\phi_{1}\left(\tilde{\hat{c}}\right)}[\1_{\{k\}}\kappa_{\hat{d}}^{p^{m}}]\right|=h_{\hat{c}}^{m}\left(\tilde{c}^{\prime}\right)\leq\varepsilon^{m}.
\end{align*}
On the other hand 
\begin{align*}
 & \sum_{k}\sum_{\hat{d}=0}^{p^{m}}\left|K^{m}v\left(k,\hat{d}|\hat{c}\right)-\E_{\phi_{1}\left(\tilde{\hat{c}}\right)}[\1_{\{k\}}\kappa_{\hat{d}}^{p^{*}}]\right|\\
 & =\sum_{k}\sum_{\hat{d}=0}^{p^{m}}\left|\frac{1}{K^{m}v(\hat{c})}\sum_{c'\in C'}K^{m}v(c',\hat{c})K^{m}v\left(k,\hat{d}|\hat{c},c{}^{\prime}\right)-\E_{\phi_{1}\left(\tilde{\hat{c}}\right)}[\1_{\{k\}}[\kappa_{\hat{d}}^{p^{m}}]\right|\\
 & \leq\sum_{c'\in C'}\frac{K^{m}v(c',\hat{c})}{K^{m}v(\hat{c})}\sum_{k}\sum_{\hat{d}=0}^{p^{m}}\left|K^{m}v\left(k,\hat{d}|\hat{c},c{}^{\prime}\right)-\E_{\phi_{1}\left(\tilde{\hat{c}}\right)}[\1_{\{k\}}[\kappa_{\hat{d}}^{p^{m}}]\right|=h_{\hat{c}}^{m}\left(\tilde{c}^{\prime}\right)\leq\varepsilon^{m}.
\end{align*}
Hence, 
\begin{align*}
 & \sum_{\hat{c}=1}^{m+1}\sum_{c^{\prime}}K^{m}v\left(\hat{c},c{}^{\prime}\right)\sum_{k,\hat{d}}\left|K^{m}v\left(k,\hat{d}|\hat{c},c{}^{\prime}\right)-K^{m}v\left(k,\hat{d}|\hat{c}\right)\right|\\
\leq & 2\varepsilon^{m}\sum_{\hat{c}=1}^{m}K^{m}v\left(\hat{c}\right)+K^{m}v\left(\hat{c}=m+1\right)\leq\delta^{m}v.
\end{align*}
Define the information structure 
\[
L^{m}v=\marg_{K\times\left\{ 0,...,p^{m}\right\} \times\left\{ 0,...,m+1\right\} }K^{m}v.
\]
Then, because $d\left(K^{m}v,v\right)=0$, Lemma \ref{lem: cond independence}
implies that 
\[
\sup_{g\in\CG}\left(\val\left(v,g\right)-\val\left(L^{m}v,g\right)\right)\leq\delta^{m}v.
\]

\emph{Proof of claim (\ref{eq:topologies 1})}. Observe that for each
$k,\hat{c},\hat{d}$, 
\[
\left(L^{m}u_{n}\right)\left(k,\hat{c},\hat{d}\right)=\E_{\tilde{u}_{n}}\left(\kappa_{\hat{c}}^{m}\left(\theta_{1}\right)\E_{\phi_{1}\left(\theta_{1}\right)}[\1_{\{k\}}\kappa_{\hat{d}}^{p^{m}}]\right).
\]
Because all the functions in the brackets above are continuous, the
weak convergence $\tilde{u}_{n}\rightarrow\tilde{u}$ implies that
\[
\left(L^{m}u_{n}\right)\left(k,\hat{c},\hat{d}\right)\rightarrow\left(L^{m}u\right)\left(k,\hat{c},\hat{d}\right).
\]
for each $k,\hat{c},\hat{d}$. Because the information structures
$L^{m}u_{n}$ and $L^{m}u$ are described on the same and finite spaces
of signals, the pointwise convergence implies 
\[
d\left(L^{m}u_{n},L^{m}u\right)\leq\left\Vert L^{m}u_{n}-L^{m}u\right\Vert \rightarrow0\text{ as \ensuremath{n\rightarrow\infty}.}
\]
Moreover, if $\hat{c}\in C^{m}$ and $\hat{d}\in D^{p^{m}}$, the
definitions imply that 
\[
\left(L^{m}u\right)\left(k,\hat{c},\hat{d}\right)\geq u\left(k,\hat{c},\hat{d}\right).
\]
Thus, 
\[
d\left(L^{m}u,u\right)\leq\|L^{m}u-u\|\leq2\left(u\left(C\backslash C^{m}\right)+u\left(D\backslash D^{p^{m}}\right)\right)\underset{n\rightarrow\infty}{\longrightarrow}0.
\]
It follows that 
\[
\delta^{m}u_{n}=\left(K^{m}u_{n}\right)\left(\hat{c}=m+1\right)\underset{n\rightarrow\infty}{\longrightarrow}\left(L^{m}u\right)\left(\hat{c}=m+1\right),
\]
and 
\[
\left(L^{m}u\right)\left(\hat{c}=m+1\right)=1-(L^{m}u)(C^{m}\times D^{p^{m}})\leq1-u(C^{m}\times D^{p^{m}})\leq u\left(C\backslash C^{m}\right)+u\left(D\backslash D^{p^{m}}\right).
\]
Together, we obtain for each $m,n$ 
\begin{align*}
 & \sup_{g\in\CG}\left(\val\left(u_{n},g\right)-\val\left(u,g\right)\right)\\
\leq & \sup_{g\in\CG}\left(\val\left(u_{n},g\right)-\val\left(L^{m}u_{n},g\right)\right)\\
 & +\sup_{g\in\CG}\left(\val\left(L^{m}u_{n},g\right)-\val\left(L^{m}u\right)\right)\\
 & +\sup_{g\in\CG}\left(\val\left(L^{m}u\right)-\val\left(u,g\right)\right)\\
\leq & \delta^{m}u_{n}+\left\Vert L^{m}u_{n}-L^{m}u\right\Vert +\left(u\left(C\backslash C^{m}\right)+u\left(D\backslash D^{p^{m}}\right)\right).
\end{align*}
Hence, 
\begin{align*}
\lim\sup_{n\rightarrow\infty}\sup_{g\in\CG}\left(\val\left(v,g\right)-\val\left(L^{m}v,g\right)\right)\leq3\left(u\left(C\backslash C^{m}\right)+u\left(D\backslash D^{p^{m}}\right)\right).
\end{align*}
When $m\rightarrow\infty$, the right hand side converges to $0$
as well.

\subsection{Proof of Proposition \ref{prop: approximate knowledge}}

Let $u^{\prime}\in\Delta\left(K\times(K_{C}\times C)\times(K_{D}\times D)\right)$
be defined so that $u=\marg_{K\times c\times D}u{}^{\prime}$ and
$u^{\prime}(\left\{ k_{C}=\kappa\left(c\right),k_{D}=\kappa\left(d\right)\right\} )=1$.
Because $u^{\prime}$ does not have any new information, we verify
(for instance using Lemma \ref{prop: joint info}) that $d\left(u,u^{\prime}\right)=0$.
We are going to show that $C$ is $16\varepsilon$-conditionally independent
from $K\times K_{D}$ given $K_{C}$. Notice that because $u$ exhibits
$\varepsilon$-knowledge, 
\begin{align*}
u^{\prime}\left\{ k_{C}\neq k\text{ or }k_{D}\neq k\right\}  & \leq u^{\prime}\left\{ k_{C}\neq k\right\} +u^{\prime}\left\{ k_{D}\neq k\right\} \\
 & \leq2\varepsilon+2\varepsilon=4\varepsilon.
\end{align*}
Thus, 
\begin{align*}
 & \sum_{k,k_{C},k_{D}}u{}^{\prime}\left(k_{C}\right)\sum_{c}\left|u^{\prime}\left(k,k_{D},c|k_{C}\right)-u{}^{\prime}\left(k,k_{D}|k_{C}\right)u{}^{\prime}\left(c|k_{C}\right)\right|\\
= & \sum_{k,k_{C},k_{D}}u{}^{\prime}\left(k,k_{C},k_{D}\right)\sum_{c}\left|u^{\prime}\left(c|k,k_{C},k_{D}\right)-\sum_{k^{\prime},k_{D}{}^{\prime}}u{}^{\prime}\left(c|k^{\prime},k_{C},k{}_{D}^{\prime}\right)u{}^{\prime}\left(k^{\prime},k{}_{D}^{\prime}|k_{C}\right)\right|\\
\leq & \sum_{k}u{}^{\prime}\left(k,k,k\right)\sum_{c}\left|u^{\prime}\left(c|k,k,k\right)-\sum_{k^{\prime},k_{D}{}^{\prime}}u{}^{\prime}\left(c|k^{\prime},k_{C}=k,k{}_{D}^{\prime}\right)u{}^{\prime}\left(k^{\prime},k{}_{D}^{\prime}|k_{C}=k\right)\right|\\
 & +2u^{\prime}\left\{ k_{C}\neq k\text{ or }k_{D}\neq k\right\} \\
\leq & \sum_{k}u{}^{\prime}\left(k,k,k\right)\sum_{c}\left|u^{\prime}\left(c|k,k,k\right)-u{}^{\prime}\left(c|k,k,k\right)\frac{u^{\prime}\left(k,k,k\right)}{u^{\prime}\left(k_{C}=k\right)}\right|\\
 & +\sum_{k}u{}^{\prime}\left(k,k,k\right)\sum_{c}\sum_{k^{\prime}\neq k\text{, or }k{}_{D}^{\prime}\neq k}\left|u^{\prime}\left(c|k,k_{C}=k,k{}_{D}\right)u{}^{\prime}\left(k^{\prime},k{}_{D}^{\prime}|k_{C}=k\right)\right|\\
 & +2u^{\prime}\left\{ k_{C}\neq k\text{ or }k_{D}\neq k\right\} \\
\leq & \sum_{k}u{}^{\prime}\left(k,k,k\right)\left|1-\frac{u^{\prime}\left(k,k,k\right)}{u^{\prime}\left(k_{C}=k\right)}\right|+3u^{\prime}\left\{ k_{C}\neq k\text{ or }k_{D}\neq k\right\} \\
\leq & \sum_{k}\left|u^{\prime}\left(k_{C}=k\right)-u^{\prime}\left(k,k,k\right)\right|+3u^{\prime}\left\{ k_{C}\neq k\text{ or }k_{D}\neq k\right\} \\
\leq & 4u^{\prime}\left\{ k_{C}\neq k\text{ or }k_{D}\neq k\right\} \leq16\varepsilon.
\end{align*}

Because an analoguous result applies to the information of the other
player, Lemma \ref{prop: joint info} shows that 
\[
d\left(u^{\prime},v^{\prime}\right)\leq16\varepsilon,
\]
where $v^{\prime}=\marg_{K\times K_{C}\times K_{D}}$. Because 
\begin{align*}
d\left(v,v^{\prime}\right)\leq & \sum_{k,k_{C},k_{D}}\left|v\left(k,k_{C},k_{D}\right)-v{}^{\prime}\left(k,k_{C},k_{D}\right)\right|\\
\leq & 2v^{\prime}\left\{ k_{C}\neq k\text{ or }k_{D}\neq k\right\} =2u{}^{\prime}\left\{ k_{C}\neq k\text{ or }k_{D}\neq k\right\} \leq4\varepsilon,
\end{align*}
the triangle inequality implies that 
\[
d\left(u,v\right)\leq d\left(u,u^{\prime}\right)+d\left(u^{\prime},v^{\prime}\right)+d\left(v,v^{\prime}\right)\leq20\varepsilon.
\]

\section{Proof of Theorem \ref{thm3}.}

$N$ is a very large even integer to be fixed later, and we write
$A=C=D=\{1,...,N\}$, with the idea of using $C$ while speaking of
actions or signals of player 1, and using $D$ while speaking of actions
and signals of player 2. We fix $\varepsilon$ and $\alpha$, to be
used later, such that 
\[
0<\varepsilon<\frac{1}{10(N+1)^{2}},\text{ and }\alpha=\frac{1}{25}.
\]
We will consider a Markov chain on $A$, satisfying:

$\bullet$ the law of the first state of the Markov chain is uniform
on $A$,

$\bullet$ for each $a$ in $A$, there are exactly $N/2$ elements
$b$ in $A$ such that $\nu(b|a)=2/N$ : given that the current state
of the Markov chain is $a$, the law of the next state is uniform
on a subset of states of size $N/2$,

$\bullet$ and few more conditions, to be be defined later.

The rest of the proof is split in 4 parts: we first define the information
structures $(u^{l})_{l\geq1}$ and some payoff structures $(g^{p})_{p\geq1}$.
Next, we show that the Markov property of the construction ensures
that the hierarchies of beliefs are preserved. Then we define two
conditions $UI1$ and $UI2$ on the information structures and show
that they imply the conclusions of theorem \ref{thm3}. Finally, we
show, via the probabilistic method, the existence of a Markov chain
$\nu$ satisfying all our conditions.

\subsection{Information and payoff structures $(u^{l})_{l\protect\geq1}$ and
$(g^{l})_{l\protect\geq1}$.}

For $l\geq1$, define the information structure $u^{l}\in\Delta(K\times C^{l}\times D^{l})$
by: for each state $k$ in $K$, signal $c=(c_{1},...,c_{l})$ in
$C^{l}$ of player 1 and signal $d=(d_{1},...,d_{l})$ in $D^{l}$
for player 2, 
\[
u^{l}(k,c,d)=\nu(c_{1},d_{1},c_{2},d_{2},...,c_{l},d_{l})\left(\frac{c_{1}}{N+1}\mathbf{1}_{k=1}+\frac{1-c_{1}}{N+1}\mathbf{1}_{k=0}\right).
\]
The following interpretation of $u^{l}$ holds: first select $(a_{1},a_{2},...,a_{2l})=(c_{1},d_{1},...,c_{l},d_{l})$
in $A^{2l}$ according to the Markov chain $\nu$ (i.e. uniformly
among the nice sequences of length $2l$), then tell $(c_{1},c_{2},...,c_{l})$
(the elements of the sequence with odd indices) to player 1, and $(d_{1},d_{2},...,d_{l})$
(the elements of the sequence with even indices) to player 2. Finally
choose the state $k=1$ with probability $c_{1}/(N+1)$, and state
$k=0$ with the complement probability $1-c_{1}/(N+1)$.

Notice that the definition is not symmetric among players, the first
signal $c_{1}$ of player 1 is uniformly distributed and plays a particular
role. The marginal of $u^{l}$ on $K$ is uniform, and the marginal
of $u^{l+1}$ over $(K\times C^{l}\times V^{l})$ holds : condition
${\it 2)}$ of theorem \ref{thm3} is satisfied.

A sequence $(a_{1},...,a_{l})$ of length $l\geq1$ is said to be
\textit{nice} it it is in the support of the Markov chain: $\nu(a_{1},...,a_{l})>0$.
For instance any sequence of length 1 is nice, and $N^{2}/2$ sequences
of length 2 are nice. Consider a sequence $(a_{1},...,a_{l})$ of
elements of $A$ which is not nice, i.e. such that $\nu(a_{1},...,a_{l})=0$.
We say that the sequence is\emph{ not nice because of player 1} if
$\min\{t\in\{1,...,l\},\nu(a_{1},...,a_{t})=0\}$ is odd, and \emph{not
nice because of player 2} if $\min\{t\in\{1,...,l\},\nu(a_{1},...,a_{t})=0\}$
is even. A sequence $(a_{1},...,a_{l})$ is now either nice, or not
nice because of player 1, or not nice because of player 2. A sequence
of length 2 is either nice, or not nice because of player 2.

For $p\geq1$, define the payoff structure $g^{p}:K\times C^{p}\times D^{p-1}\to[-1,1]$
such that for all $k$ in $K$, $c'=(c'_{1},...,c'_{p})$ in $C^{p}$,
$d'=(d'_{1},...,d'_{p-1})$ in $D^{p-1}$ : 
\begin{eqnarray*}
g^{p}(k,c',d') & = & g_{0}(k,c'_{1})+h^{p}(c',d'),\;\;{\mathnormal{w}ith}\\
g_{0}(k,c'_{1}) & = & -{\left(k-\frac{u'_{1}}{N+1}\right)}^{2}+\frac{N+2}{6(N+1)},\;\\
h^{p}(c',d') & = & \left\{ \begin{array}{ccl}
\varepsilon & \mbox{if} & (c'_{1},d'_{1},...,c'_{p})\;{is}\;{\it nice},\\
5\varepsilon & \mbox{if} & (c'_{1},d'_{1},...,c'_{p})\;{is}\;not\;{\it nice}\;because\;of\;player\;2,\\
-5\varepsilon & \mbox{if} & (c'_{1},d'_{1},...,c'_{p})\;{is}\;not\;{\it nice}\;because\;of\;player\;1.
\end{array}\right.
\end{eqnarray*}

One can check that $|g^{p}|\leq5/6+5\varepsilon\leq8/9$. Regarding
the $g_{0}$ part of the payoff, consider a decision problem for player
1 where: $c_{1}$ is selected uniformly in $A$ and the state is selected
to be $k=1$ with probability $c_{1}/(N+1)$ and $k=0$ with probability
$1-c_{1}/(N+1)$. Player 1 observes $c_{1}$ but not $k$, and he
choose $c'_{1}$ in $A$ and receive payoff $g_{0}(k,c'_{1})$. We
have $\frac{c_{1}}{N+1}g_{0}(1,c'_{1})+(1-\frac{c_{1}}{N+1})g_{0}(0,c'_{1})$
$=$ $\frac{1}{(N+1)^{2}}(c'_{1}(2c_{1}-c'_{1})+(N+1)((N+2)/6-c_{1}))$.
To maximize this expected payoff, it is well known that player 1 should
play his belief on $k$, i.e. $c'_{1}=c_{1}$. Moreover, if player
1 chooses $c'_{1}\neq c_{1}$, its expected loss from not having chosen
$c_{1}$ is at least $\frac{1}{(N+1)^{2}}\geq10\varepsilon$. And
the constant $\frac{N+2}{6(N+1)}$ has been chosen such that the value
of this decision problem is 0.

Consider now $l\geq1$ and $p\geq1$. By definition, the Bayesian
game $\Gamma(u^{k},g^{p})$ is played as follows: first, $(c_{1},d_{1},...,c_{l},d_{l})$
is selected according to the law $\nu$ of the Markov chain, player
1 learns $(c_{1},...,c_{l})$, player 2 learns $(d_{1},...,d_{l})$
and the state is $k=1$ with probability $c_{1}/(N+1)$ and $k=0$
otherwise. Then \textit{simultaneously} player 1 chooses $c'$ in
$C^{p}$ and player 2 chooses $d'$ in $D^{p-1}$, and finally the
payoff to player 1 is $g^{p}(k,c',d')$. Notice that by the previous
paragraph about $g_{0}$, it is always strictly dominant for player
1 to report correctly his first signal, i.e. to choose $c'_{1}=c_{1}$.
We will show in the next section that if $l\geq p$ and player 1 simply
plays the sequence of signals he received, player 2 can not do better
than also reporting truthfully his own signals, leading to a value
not lower than the payoff for nice sequences, that is $\varepsilon$.
On the contrary in the game $\Gamma(u^{l},g^{l+1})$, player 1 has
to report not only the $l$ signals he has received, but also an extra-signal
$c'_{l+1}$ that he has to guess. In this game we will prove that
if player 2 truthfully reports his own signals, player 1 will incur
the payoff $-5\varepsilon$ with probability at least (approximately)
1/2, and this will result in a low value. These intuitions will prove
correct in the next section, under some conditions $UI1$ and $UI2$.

\subsection{Higher order beliefs\label{subsec:Higher-order-beliefs}}

Recall that $n$-order beliefs are defined inductively as conditional
laws. Precisely, the first order beliefs $\theta_{1}^{i}$ of player
$i$ is the conditional law of $k$ given her signal. The $n$-order
belief $\theta_{n}^{i}$ of player $i$ is the conditional law of
$(\omega,\theta_{n-1}^{-i})$ given her signal. In this construction,
conditional laws are seen as random variables taking values in space
of probability measures. 
\begin{lem}
\label{lem_beliefs} For all $l>p$, the joint distribution of $(\omega,\theta_{2p}^{1},\theta_{2p}^{2})$
induced by the information structure $u^{l}$ is independent of $l$. 
\end{lem}

\begin{proof}
We use the notation ${\mathcal{L}}(X|Y)$ for the conditional law
of $X$ given $Y$, and the identification $(a_{1},...,a_{2l})=(c_{1},d_{1},....,c_{l},d_{l})$.
At first, note that by construction $k$ and $(a_{2},....,a_{2l})$
are conditionally independent given $a_{1}$, so that the sequence
$(k,a_{1},a_{2},...,a_{2l})$ is a Markov process. It follows that
$\theta_{1}^{1}={\mathcal{L}}(k|c_{1},...,c_{l})$ =${\mathcal{L}}(k|c_{1}).$
The Markov property implies that 
\[
\theta_{1}^{2}={\mathcal{L}}(k|d_{1},....,d_{l})={\mathcal{L}}(k|d_{1}),\;\theta_{2}^{2}={\mathcal{L}}(d,\theta_{1}^{1}(c_{1})|d_{1},....,d_{l})={\mathcal{L}}(k,\theta_{1}^{1}(c_{1})|d_{1}),
\]
and therefore we have 
\[
\theta_{2}^{1}={\mathcal{L}}(k,\theta_{1}^{2}(d_{1})|c_{1},....,c_{l})={\mathcal{L}}(k,\theta_{1}^{2}(d_{1})|c_{1},c_{2}).
\]
By induction, and applying the same argument (future and past of a
Markov process are conditionally independent given the current position),
we deduce that for all $n\geq1$, 
\[
\theta_{2n}^{1}={\mathcal{L}}(k,\theta_{2n-1}^{2}|c_{1},....,c_{\min(l,n+1)}),\;\;\;\theta_{2n+1}^{1}={\mathcal{L}}(k,\theta_{2n}^{2}|c_{1},....,c_{\min(l,n+1)}),
\]
\[
\theta_{2n-1}^{2}={\mathcal{L}}(k,\theta_{2n-2}^{1}|d_{1},....,d_{\min(l,n)}),\;\;\;\theta_{2n}^{2}={\mathcal{L}}(k,\theta_{2n-1}^{1}|d_{1},....,d_{\min(l,n)}).
\]
As a consequence, for all $n\leq p$, these conditional laws do not
depend on which $u^{l}$ we are using as soon as $l>p$. 
\end{proof}

\subsection{Conditions UI and values}

To prove that the intuitions of the previous paragraph are correct,
we need to ensure that players have incentives to report their true
signals, so we need additional assumptions on the Markov chain. \\

\noindent \textbf{Notations and definition:} Let $l\geq1$, $m\geq0$,
$c=(c_{1},...,c_{l})$ in $C^{l}$ and $d=(d_{1},...,d_{m})$ in $D^{m}$.
We write :

\centerline{%
\begin{tabular}{lcll}
$a^{2q}(c,d)$  & =  & $(c_{1},d_{1},....,c_{q},d_{q})\in A^{2q}$  & {for each } $q\leq\min\{l,m\}$, \tabularnewline
$a^{2q+1}(c,d)$  & =  & $(c_{1},d_{1},....,c_{q},d_{q},c_{q+1})\in A^{2q+1}$  & { for each }$q\leq\min\{l-1,m\}$. \tabularnewline
\end{tabular}}

\noindent For $r\leq\min\{2l,2m+1\}$,

\centerline{we say that $c$ and $d$ are \emph{nice at level }$r$,
and we write $c\smile_{r}d,$ if $a^{r}(c,d)$ is nice.}

%Define also the first level at which $c$ and $d$ are not nice by: 
%\begin{align*}
%f(c,d)& =\min \{r\leq \min \left\{ 2l,2m+1\right\} ,\;\text{not }c\smile
%_{r}d\},\text{ and } \\
%f(c,d)& =+\infty \text{ if }c\smile _{\min
%\{2l,2m+1\}}d \; (\text{i.e. if } a^{\min \left\{ 2l,2m+1\right\}}(c,d) \text{ is nice}).
%\end{align*}

%Function $f$ identifies the first level at which $u$ and $v$ are not
%consistent.
%%
%\begin{remark}
%For any $u$ and $v,$ $u\smile_{1}v$ since $C^1=A$. By 
%\eqref{consistency
%criterion}, $u\smile_{r}v$ implies $u\smile_{r-1}v$. And for $u$ and $v$ in $%
%A^{l}$ : 
%\begin{equation*}
%u\smile_{2l}v\Longleftrightarrow a^{2l}(u,v)\in C^{2l}\Longleftrightarrow
%f(u,v)=+\infty.
%\end{equation*}
%\end{remark}
\vspace{0.5cm}

In the next definition we consider an information structure $u^{l}\in\Delta(K\times C^{l}\times D^{l})$
and denote by ${\tilde{c}}$ and ${\tilde{d}}$ the respective random
variables of the signals of player 1 and 2. 
\begin{defn}
\label{defUI}We say that the \emph{conditions $UI1$ are satisfied
}if for all $l\geq1$, all ${c}=({c}_{1},...,{c}_{l})$ in $C^{l}$
and $c^{\prime}=(c_{1}^{\prime},...,c_{l+1}^{\prime})$ in $C^{l+1}$
such that ${c}_{1}=c_{1}^{\prime}$, we have 
\begin{equation}
u^{l}\left(c^{\prime}\smile_{2l+1}\tilde{d}\;\big|\;\tilde{c}={c},c^{\prime}\smile_{2l}\tilde{d}\right)\in[1/2-\alpha,1/2+\alpha]\label{eq61}
\end{equation}
and for all $m\in\{1,...,l\}$ such that ${c}_{m}\neq c_{m}^{\prime}$,
for $r=2m-2,2m-1$, 
\begin{equation}
u^{l}\left(c^{\prime}\smile_{r+1}\tilde{d}\;\big|\;\tilde{c}=c,c^{\prime}\smile_{r}\tilde{d}\right)\in[1/2-\alpha,1/2+\alpha].\label{eq62}
\end{equation}
\end{defn}

We say that the\emph{ conditions $UI2$ are satisfied} if for all
$1\leq p\leq l$, for all ${d}\in D^{l}$, for all $d^{\prime}\in D^{p-1}$,
for all $m\in\{1,...,p-1\}$ such that ${d}_{m}\neq d_{m}^{\prime}$,
for $r=2m-1,2m$ 
\begin{equation}
u^{l}\left(\tilde{c}\smile_{r+1}d^{\prime}|\tilde{d}={d},\tilde{c}\smile_{r}d^{\prime}\right)\in[1/2-\alpha,1/2+\alpha].\label{eq63}
\end{equation}

To understand the conditions $UI1$, consider the Bayesian game $\Gamma(u^{l},g^{l+1})$,
and assume that player 2 truthfully reports his sequence of signals
and that player 1 has received the signals $(c_{1},...,c_{l})$ in
$C^{l}$. (\ref{eq61}) states that if the sequence of reported signals
$(c'_{1},\tilde{d}_{1},...,c'_{l},\tilde{d}_{l})$ is nice at level
$2l$, then whatever the last reported signal $c'_{l+1}$, the conditional
probability that $(c'_{1},\tilde{d}_{1},...,c'_{l},\tilde{d}_{l},c'_{l+1})$
is still nice is in $[1/2-\alpha,1/2+\alpha]$, i.e. close to 1/2.
Regarding (\ref{eq62}), first notice that if $c'=c$, then by construction
$(c'_{1},\tilde{d}_{1},...,c'_{l},\tilde{d}_{l})$ is nice and $u^{l}\left(c^{\prime}\smile_{r+1}\tilde{d}\;\big|\;\tilde{c}=c,c^{\prime}\smile_{r}\tilde{d}\right)=u^{l}\left(c\smile_{r+1}\tilde{d}\;\big|\;\tilde{c}=c\right)=1$
for each $r=1,...,2l-1$. Assume now that for some $m=1,...,l$, player
1 misreports his $m^{th}$-signal, i.e. reports $c'_{m}\neq c_{m}$.
(\ref{eq62}) requires that given that the reported signals were nice
so far (at level $2m-2$), the conditional probability that the reported
signals are not nice at level $2m-1$ (integrating $c'_{m}$) is close
to 1/2, and moreover if the reported signals are nice at this level
$2m-1$, adding the next signal $\tilde{d}_{m}$ of player 2 has probability
close to 1/2 to keep the reported sequence nice. Conditions $UI2$
have a similar interpretation, considering the Bayesian games $\Gamma(u^{l},g^{p})$
for $p\leq l$, assuming that player 1 reports truthfully his signals
and that player 2 plays $d'$ after having received the signals $d$. 
\begin{prop}
\label{pro2} Conditions $UI1$ and $UI2$ imply : 
\begin{eqnarray}
\forall l\geq1,\forall p\in\{1,...,l\}, & \val(u^{l},g^{p})\geq\varepsilon.\label{eq64}\\
\forall l\geq1, & \val(u^{l},g^{l+1})\leq-\varepsilon.\label{eq66}
\end{eqnarray}
\end{prop}

As a consequence of this proposition, under conditions $UI1$ and
$UI2$ we easily obtain condition $1)$ of theorem \ref{thm3} : 
\begin{cor}
\label{cor2} If $l\neq p$ then $d(u^{l},u^{p})\geq2\varepsilon.$ 
\end{cor}

\begin{proof}
Assume $l>p$, then $d(u^{l},u^{p})\geq\val(u^{l},g^{p+1})-\val(u^{p},g^{p+1})\geq\varepsilon-(-\varepsilon).$ 
\end{proof}
\vspace{0.5cm}

\noindent \textbf{Proof of proposition \ref{pro2}.} We assume that
$UI1$ and $UI2$ hold. We fix $l\geq1$, work on the probability
space $K\times C^{l}\times D^{l}$ equipped with the probability $u^{l}$,
and denote by $\tilde{c}$ and $\tilde{d}$ the random variables of
the signals received by the players.

1) We first prove (\ref{eq64}), and consider the game $\Gamma(u^{l},g^{p})$
with $p\in\{1,...,l\}$. We assume that player 1 chooses the truthful
strategy. Fix $d=(d_{1},...,d_{l})$ in $D^{l}$ and $d'=(d'_{1},...,d'_{p-1})$
in $D^{p-1}$, and assume that player 2 has received the signal $d$
and chooses to report $d'$.

%We work on the probability space $K \times C^l \times D^l$ equipped
%with the probability $u^l$, and denote by $\tilde{c}$ and $\tilde{d}$ the random variables of the signals received by the players.%Recall $t^{\ast}=t\left(
%v,v^{\prime}\right)  =\min\left\{  s:v_{s}\neq v_{s}^{\prime}\right\}  -1\leq
%p-1$ is the last level at which the two reports agree. If $t^{\ast}=p-1$, the
%report of the minimizer is truthul, the reports are consistent at level
%$2p-1,$ there is no mistake, and the payoff from the second component is equal
%to $\varepsilon$.

%If the report is not truthful, $t^{\ast}<p-1,$ it might not be consistent with
%the truthfully reported type of the maximizer, and there is a possibility of a
%mistake. 

Define the non-increasing sequence of events: 
\[
A_{n}=\{\tilde{c}\smile_{n}d'\}.
\]
We will prove by backward induction that: 
\begin{equation}
\forall n=1,...,p,\;\;\mathbb{E}[h^{p}(\tilde{c},d')|\tilde{d}=d,A_{2n-1}]\geq\varepsilon.\label{eq65}
\end{equation}

If $n=p$, $h^{p}(\tilde{c},d')=\varepsilon$ on the event $A_{2p-1}$,
implying the result. Assume now that for some $n$ such that $1\leq n<p$,
we have : $\mathbb{E}[h^{p}(\tilde{c},d')|\tilde{d}=d,A_{2n+1}]\geq\varepsilon.$
Since we have a non-increasing sequence of events, $\1{}_{A_{2n-1}}=\1_{A_{2n+1}}+\1_{A_{2n-1}}\1_{A_{2n}^{c}}+\1_{A_{2n}}\1_{A_{2n+1}^{c}},$
so by definition of the payoffs, $h^{p}(\tilde{c},d')\1_{A_{2n-1}}=h^{p}(\tilde{c},d')\1_{A_{2n+1}}+5\varepsilon\1_{A_{2n-1}}\1_{A_{2n}^{c}}-5\varepsilon\1_{A_{2n}}\1_{A_{2n+1}^{c}}.$

First assume that $d'_{n}=d_{n}$. By construction of the Markov chain,
$u^{l}(A_{2n+1}|A_{2n-1},\tilde{d}=d)=1$, implying that $u^{l}(A_{2n+1}^{c}|A_{2n-1},\tilde{d}=d)=u^{l}(A_{2n}^{c}|A_{2n-1},\tilde{d}=d)=0$.
As a consequence, 
\begin{eqnarray*}
\mathbb{E}[h^{p}(\tilde{c},d')|\tilde{d}=d,A_{2n-1}] & = & \mathbb{E}[h^{p}(\tilde{c},d')\1_{A_{2n+1}}|\tilde{d}=d,A_{2n-1}]\\
 & = & \mathbb{E}[\mathbb{E}[h^{p}(\tilde{c},d')|\tilde{d}=d,A_{2n+1}]\1_{A_{2n+1}}|\tilde{d}=d,A_{2n-1}]\\
 & \geq\varepsilon.
\end{eqnarray*}
Assume now that $d'_{n}\neq d_{n}$. Assumption UI2 implies that :
\begin{eqnarray*}
u^{l}(A_{2n}^{c}|A_{2n-1},\tilde{d}=d) & \geq & 1/2-\alpha,\\
u^{l}(A_{2n}\cap A_{2n+1}^{c}|A_{2n-1},\tilde{d}=d) & \leq & (1/2+\alpha)^{2},\\
u^{l}(A_{2n+1}|A_{2n-1},\tilde{d}=d) & \geq & (1/2-\alpha)^{2}.
\end{eqnarray*}
It follows that : 
\begin{align*}
\mathbb{E}[h^{p}(\tilde{c},d'|\tilde{d})=d,A_{2n-1}] & =\mathbb{E}[\mathbb{E}[h^{p}(\tilde{c},d')|\tilde{d}=d,A_{2n+1}]\1_{A_{2n+1}}|\tilde{d}=d,A_{2n-1}]\\
 & \quad+5\varepsilon u^{l}(A_{2n}^{c}|A_{2n-1},\tilde{d}=d)-5\varepsilon u^{l}(A_{2n}\cap A_{2n+1}^{c}|A_{2n-1},\tilde{d}=d)\\
 & \geq\varepsilon\,(\frac{1}{4}-\alpha+\alpha^{2})+5\,\varepsilon\,(\frac{1}{2}-\alpha)-5\,\varepsilon\,(\frac{1}{4}+\alpha+\alpha^{2})\\
 & =\varepsilon\,(\frac{3}{2}-11\alpha-4\alpha^{2})\geq\varepsilon,
\end{align*}
And (\ref{eq65}) follows by backward induction.

Since $A_{1}$ is an event which holds almost surely, we deduce that
$\mathbb{E}[h^{p}(\tilde{c},d')|\tilde{d}=d]\geq\varepsilon.$ Hence
the truthful strategy of player 1 guarantees the payoff $\varepsilon$
in $\Gamma(u^{l},g^{p})$.\\

2) We now prove (\ref{eq66}) and consider the Bayesian game $\Gamma(u^{l},g^{l+1})$,
assuming that player 2 chooses the truthful strategy. Fix $c=(c_{1},...,c_{l})$
in $C^{l}$ and $c'=(c'_{1},...,c'_{l-1})$ in $C^{l-1}$, and assume
that player 1 has received the signal $c$ and chooses to report $c'$.
We will show that the expected payoff of player 1 is not larger than
$-\varepsilon$, and assume w.l.o.g. that $c'_{1}=c_{1}$. Consider
the non-increasing sequence of events : 
\[
B_{n}=\{c^{\prime}\smile_{n}\tilde{d}\,\}.
\]
We will prove by backward induction that: 
\[
\forall n=1,...,l,\;\;\mathbb{E}[h^{l+1}(c^{\prime},\tilde{d})|\tilde{c}=c,B_{2n}]\leq-\varepsilon.
\]

If $n=l$, we have $\1_{B_{2l}}=\1_{B_{2l+1}}+\1_{B_{2l}}\1_{B_{2l+1}^{c}}$,
and $h^{l+1}(c^{\prime},\tilde{d})\1_{B_{2l}}=\varepsilon\1_{B_{2l+1}}-5\varepsilon\1_{B_{2l}}\1_{B_{2l+1}^{c}}.$
UI1 implies that $|u^{l}(B_{2l+1}|\tilde{c}=c,B_{2l})-\frac{1}{2}|\leq\alpha$
%\end{equation*}
%\begin{equation*}
%|u^l( B_{2l+1}^c | \tilde{c}=c, B_{2l} ) - \frac{1}{2} | \leq \alpha .
%\end{equation*}
, and it follows that : 
\begin{align*}
\mathbb{E}[h^{l+1}(c^{\prime},\tilde{d})|\tilde{c}=c,B_{2l}] & =\varepsilon\,u^{l}(B_{2l+1}|\tilde{c}=c,B_{2l})-5\varepsilon\,u^{l}(B_{2l+1}^{c}|u=\hat{u},B_{2l})\\
 & \leq\varepsilon\,(\frac{1}{2}+\alpha)-5\varepsilon\,(\frac{1}{2}-\alpha)\leq-\varepsilon.
\end{align*}
%\\
%& =\varepsilon (-2+6\alpha )\leq %
%where we used that $\alpha+\alpha^2 \leq \frac{1}{22}$ implies that $\alpha
%\leq \frac{1}{6}$.

Assume now that for some $n=1,...,l-1$, we have $\mathbb{E}[h^{l+1}(c^{\prime},\tilde{d})|\tilde{c}=c,B_{2n+2}]\leq-\varepsilon.$
We have $\1_{B_{2n}}=\1_{B_{2n+2}}+\1_{B_{2n}}\1_{B_{2n+1}^{c}}+\1_{B_{2n+1}}\1_{B_{2n+2}^{c}},$
and by definition of $h^{l+1}$, 
\[
h^{l+1}(c^{\prime},\tilde{d})\1_{B_{2n}}=h^{l+1}(c^{\prime},\tilde{d})\1_{B_{2n+2}}-5\varepsilon\1_{B_{2n}}\1_{B_{2n+1}^{c}}+5\varepsilon\1_{B_{2n+1}}\1_{B_{2n+2}^{c}}.
\]
First assume that $c_{n+1}^{\prime}=c_{n+1}$, then $u^{l}(B_{2n+2}|B_{2n},\tilde{c}=c)=1$.
Then : 
\begin{eqnarray*}
\mathbb{E}[h^{l+1}(c^{\prime},\tilde{d})|\tilde{c}=c,B_{2n}] & = & \mathbb{E}[h^{l+1}(c^{\prime},\tilde{d})\1_{B_{2n+2}}|\tilde{c}=c,B_{2n}],\\
 & = & \mathbb{E}[\mathbb{E}[h^{l+1}(c^{\prime},\tilde{d})|\tilde{c}=c,B_{2n+2}]\1_{B_{2n+2}}|\tilde{c}=c,B_{2n}]\leq-\varepsilon.
\end{eqnarray*}

Assume on the contrary that $c_{n+1}^{\prime}\neq c_{n+1}$, assumption
UI1 implies that : 
\begin{eqnarray*}
u^{l}(B_{2n+1}^{c}|B_{2n},\tilde{c}=c) & \geq & 1/2-\alpha,\\
u^{l}(B_{2n+1}\cap B_{2n+2}^{c}|B_{2n},\tilde{c}=c) & \leq & (1/2+\alpha)^{2},\\
u^{l}(B_{2n+2}|B_{2n},\tilde{c}=c) & \geq & (1/2-\alpha)^{2}.
\end{eqnarray*}
It follows that : 
\begin{align*}
\mathbb{E}[h^{l+1}(c^{\prime},\tilde{d})|\tilde{c}=c,B_{2n}] & =\mathbb{E}[\mathbb{E}[h^{l+1}(c^{\prime},\tilde{d})|\tilde{c}=c,B_{2n+2}]\1_{B_{2n+2}}|\tilde{c}=c,B_{2n}]\\
 & \quad-5\,\varepsilon\,u^{l}(B_{2n+1}^{c}|B_{2n},\tilde{c}=c)+5\,\varepsilon\,u^{l}(B_{2n+1}\cap B_{2n+2}^{c}|B_{2n},\tilde{c}=c)\\
 & \leq-\,\varepsilon\,(\frac{1}{4}-\alpha+\alpha^{2})-5\,\varepsilon\,(\frac{1}{2}-\alpha)+5\,\varepsilon\,(\frac{1}{4}+\alpha+\alpha^{2})\leq-\varepsilon.
\end{align*}

By induction, we obtain $\mathbb{E}[h^{l+1}(c^{\prime},\tilde{d})|\tilde{c}=c,B_{2}]\leq-\varepsilon$.
Since $B_{2}$ holds almost surely here, we get $\mathbb{E}[h^{l+1}(c^{\prime},\tilde{d})|\tilde{c}=c]\leq-\varepsilon,$
showing that the truthful strategy of player 2 guarantees that the
payoff of the maximizer is less or equal to $-\varepsilon$, and concluding
the proof.

\subsection{Existence of an appropriate Markov chain\label{subsec:Existence-of-Markov chain}\label{existence}}

Here we conclude the proof of Theorem \ref{thm3} by showing the existence
of an even integer $N$ and a Markov chain with law $\nu$ on $A=\{1,...,N\}$
satisfying our conditions :

$1)$ the law of the first state of the Markov chain is uniform on
$A$,

$2)$ for each $a$ in $A$, there are exactly $N/2$ elements $b$
in $A$ such that $\nu(b|a)=2/N$,

$3)$ $UI1$ and $UI2$.

Denoting by $P=(P_{a,b})_{(a,b)\in A^{2}}$ the transition matrix
of the Markov chain, we have to prove the existence of $P$ satisfying
$2)$ and $3)$. The proof is non constructive and uses the following
probabilistic method, where we select independently for each $a$
in $A$, the set $\{b\in A,P_{a,b}>0\}$ uniformly among the subsets
of $A$ with cardinal $N/2$. We will show that when $N$ goes to
infinity, the probability of selecting an appropriate transition matrix
does not only become positive, but converges to 1.

Formally, denote by $\mathcal{S}_{A}$ the collection of all subsets
$S\subseteq A$ with cardinality $\left\vert S\right\vert =\frac{1}{2}N$.
We consider a collection $\left(S_{a}\right)_{a\in A}$ of i.i.d.
random variables uniform distributed over $\mathcal{S}_{A}$ defined
on a probability space $(\Omega_{N},\mathcal{F}_{N},\mathbb{P}_{N})$.
For all $a$, $b$ in $A$, let 
\[
X_{a,b}=\1_{\{b\in S_{a}\}}\;{\rm and}\;P_{a,b}=\frac{2}{N}X_{a,b}.
\]
By construction, $P$ is a transition matrix satisfying $2)$. Theorem
\ref{thm3} will now follow directly from the following proposition. 
\begin{prop}
\label{pro3} 
\[
\mathbb{P}_{N}\left(\text{ }P\text{ induces a Markov chain satisfying UI1 and UI2 }\right)\xrightarrow[n\to\infty]{}1.
\]
In particular, the above probability is strictly positive for all
sufficiently large $N$. 
\end{prop}

The rest of this section is devoted to the proof of proposition \ref{pro3}.

We start with probability bounds based on Hoeffding's inequality. 
\begin{lem}
\label{lemT1}For any $a\neq b,$ each $\gamma>0$ 
\[
\mathbb{P}_{N}\left(\left\vert |S_{a}\cap S_{b}|-\frac{1}{4}N\right\vert \geq\gamma N\right)\leq\frac{1}{2}e^{4}Ne^{-2\gamma^{2}N}.
\]
\end{lem}

\begin{proof}
Consider a family of i.i.d. Bernoulli variables $(\widetilde{X}_{i,j})_{i=a,b,\,j\in A}$
of parameter $\frac{1}{2}$ defined on a space $(\Omega,\mathcal{F},\mathbb{P})$.
For $i=a,b$, define the events $\widetilde{L}_{i}=\{\sum_{j\in A}\widetilde{X}_{i,j}=\frac{N}{2}\}$
and the set-valued variables $\widetilde{S}_{i}=\{j\in A\,|\,\widetilde{X}_{i,j}=1\}$.
It is straightforward to check that the conditional law of $(\widetilde{S}_{a},\widetilde{S}_{b})$
given $\widetilde{L}_{a}\cap\widetilde{L}_{b}$ under $\mathbb{P}$
is the same as the law of $(S_{a},S_{b})$ under $\mathbb{P}_{N}$.
It follows that 
\begin{align*}
\mathbb{P}_{N}\left(\left\vert |S_{a}\cap S_{b}|-\frac{1}{4}N\right\vert \geq\gamma N\right) & =\mathbb{P}\left(\left\vert |\widetilde{S}_{a}\cap\widetilde{S}_{b}|-\frac{1}{4}N\right\vert \geq\gamma N\,\Big|\,\widetilde{L}_{a}\cap\widetilde{L}_{b}\right)\\
 & \leq\frac{\mathbb{P}\left(\left\vert |\widetilde{S}_{a}\cap\widetilde{S}_{b}|-\frac{1}{4}N\right\vert \geq\gamma N\right)}{\mathbb{P}\left(\widetilde{L}_{a}\cap\widetilde{L}_{b}\right)}.
\end{align*}
Using Hoeffding inequality, we have 
\[
\mathbb{P}\left(\left\vert |\widetilde{S}_{a}\cap\widetilde{S}_{b}|-\frac{1}{4}N\right\vert \geq\gamma N\right)=\mathbb{P}\left(\left\vert \sum_{j\in A}\widetilde{X}_{a,j}\widetilde{X}_{b,j}-\frac{1}{4}N\right\vert \geq\gamma N\right)\leq2e^{-2\gamma^{2}N}.
\]
On the other hand, using Stirling approximation\footnote{We have $n^{n+\frac{1}{2}}e^{-n}\leq n!\leq en^{n+\frac{1}{2}}e^{-n}$
for each $n$.}, we have 
\[
\mathbb{P}\left(\widetilde{L}_{a}\cap\widetilde{L}_{b}\right)=\left(\frac{1}{2^{N}}\frac{N!}{\left(\frac{N}{2}!\right)^{2}}\right)^{2}\geq\left(\frac{2^{N+1}N^{-\frac{1}{2}}}{2^{N}e^{2}}\right)^{2}=\frac{4}{Ne^{4}}.
\]
We deduce that $\mathbb{P}_{N}\left(\left\vert |S_{a}\cap S_{b}|-\frac{1}{4}N\right\vert \geq\gamma N\right)\leq\frac{1}{2}e^{4}Ne^{-2\gamma^{2}N}.$ 
\end{proof}
\begin{lem}
\label{lemT2}For each $a\neq b,$ for any subset $S\subseteq A$
and any $\gamma\geq\frac{1}{2N-2}$, 
\[
\mathbb{P}_{N}\left(\left\vert \sum_{i\in S}X_{i,a}-\frac{1}{2}\left\vert S\right\vert \right\vert \geq\gamma N\right)\leq2e^{-2N\gamma^{2}},\;{\mathnormal{a}nd}\;\mathbb{P}_{N}\left(\left\vert \sum_{i\in S}X_{i,a}X_{i,b}-\frac{1}{4}\left\vert S\right\vert \right\vert \geq\gamma N\right)\leq2e^{-\frac{1}{2}N\gamma^{2}}.
\]
%\begin{eqnarray*}
%\mathbb{P}_N\left( \left\vert \sum_{i\in S}X_{i,a}-\frac{1}{2}\left\vert
%S\right\vert \right\vert \geq \gamma N\right) &\leq &e^{-\frac{1}{16}N\gamma
%^{2}}, \\
%\mathbb{P}_N\left( \left\vert \sum_{i\in S}X_{i,a}X_{i,b}-\frac{1}{4}%
%\left\vert S\right\vert \right\vert \geq \gamma N\right) &\leq &e^{-\frac{1}{%
%16}N\gamma ^{2}}.
%\end{eqnarray*}
\end{lem}

\begin{proof}
For the first inequality, notice that $X_{i,a}$ are i.i.d. Bernoulli
random variables with parameter $\frac{1}{2}$. The Hoeffding inequality
implies that : 
\[
\mathbb{P}_{N}\left(\left\vert \sum_{i\in S}X_{i,a}-\frac{1}{2}\left\vert S\right\vert \right\vert \geq\gamma N\right)\leq2e^{-2\gamma^{2}\frac{N^{2}}{|S|}}\leq2e^{-2N\gamma^{2}}.
\]
\end{proof}
For the second inequality, let $Z_{i}=X_{i,a}X_{i,b}.$ Notice that
all variables $Z_{i}$ are i.i.d. Bernoulli random variables with
parameter $p=\frac{1}{2}\left(\frac{\frac{N}{2}-1}{N-1}\right)=\frac{1}{4}-\frac{1}{4N-4}$.
The Hoeffding inequality implies that 
\begin{eqnarray*}
\mathbb{P}_{N}\left(\left\vert \sum_{i\in S}Z_{i}-\frac{1}{4}\left\vert S\right\vert \right\vert \geq\gamma N\right) & \leq & \mathbb{P}_{N}\left(\left\vert \sum_{i\in S}Z_{i}-p\left\vert S\right\vert \right\vert \geq\frac{1}{2}\gamma N\right)\\
 & \leq & 2e^{-2\gamma^{2}\frac{N^{2}}{|S|}}\leq2e^{-\frac{1}{2}N\gamma^{2}},
\end{eqnarray*}
where we used that $|S||p-\frac{1}{4}|\leq\frac{N}{4N-4}\leq\frac{\gamma N}{2}$
for the first inequality.

\vspace{0.3cm}

\begin{defn}
For each $a\neq b$ and $c\neq d,$ each $\gamma>0,$ define : 
\end{defn}

\begin{tabular}{lll}
$Y_{a}=2\sum_{i\in A}X_{i,a},$  & $Y^{c}=2\sum_{i\in A}X_{c,i}=N$,  & \;\tabularnewline
$Y_{a,b}=4\sum_{i\in A}X_{i,a}X_{i,b},$  & $Y_{a}^{c}=4\sum_{i\in A}X_{i,a}X_{c,i},$  & $Y^{c,d}=4\sum_{i\in A}X_{c,i}X_{d,i},$ \tabularnewline
$Y_{a,b}^{c}=8\sum_{i\in A}X_{i,a}X_{i,b}X_{c,i},$  & $Y_{a}^{c,d}=8\sum_{i\in A}X_{i,a}X_{c,i}X_{d,i},$  & $Y_{a,b}^{c,d}=16\sum_{i\in A}X_{i,a}X_{i,b}X_{c,i}X_{d,i}.$ \tabularnewline
\end{tabular}

\vspace{0.3cm}

\begin{lem}
\label{lemT3}For each $a\neq b$ and $c\neq d,$ each $\gamma\geq64/N,$
each of the variables $Z\in\{Y_{a},Y^{c},Y_{a,b},Y^{c,d},Y_{a}^{c},Y_{a,b}^{c},Y_{a}^{c,d},Y_{a,b}^{c,d}\},$
\[
\mathbb{P}_{N}\left(\left\vert Z-N\right\vert \geq\gamma N\right)\leq e^{4}Ne^{-\frac{N}{32}{(\frac{\gamma}{10})}^{2}}.
\]
\end{lem}

\begin{proof}
In case $Z=Y_{a}$ or $Y_{a,b},$ the bound follows from Lemma \ref{lemT2}
(for $S=A)$. If case $Z=Y^{c},$ the bound is trivially satisfied.
If $Z=Y^{c,d},$ the bound follows from Lemma \ref{lemT1}.

In case $Z=Y_{a,b}^{c,d}$, notice that 
\[
Y_{a,b}^{c,d}=16\sum\limits _{i\in S_{c}\cap S_{d}}Z_{i},\;\;{\rm where}\;\;Z_{i}=X_{i,a}X_{i,b}.
\]
All variables $Z_{i}$ are i.i.d. Bernouilli random variables with
parameter $p=\frac{1}{4}-\frac{1}{4N-4}$. Moreover, $\left\{ Z_{i}\right\} _{i\neq c,d}$
are independent of $S_{c}\cap S_{d}$. Up to enlarge the probability
space, we can construct a new collection of i.i.d. Bernoulli random
variables $Z_{i}^{\prime}$ such that $Z_{i}^{\prime}=Z_{i}$ for
all $i\neq c,d$ and such that $\left\{ (Z_{i}^{\prime})_{i\in A},S_{c}\cap S_{d}\right\} $
are all independent. Then, 
\[
\left\vert Y_{a,b}^{c,d}-16\sum\limits _{i\in S_{c}\cap S_{d}}Z_{i}^{\prime}\right\vert \leq32,
\]
and, because $\frac{1}{2}\gamma N\geq32,$ we have 
\[
\mathbb{P}_{N}\left(\left\vert Y_{a,b}^{c,d}-N\right\vert \geq\gamma N\right)\leq\mathbb{P}_{N}\left(\left\vert \sum\limits _{i\in S_{c}\cap S_{d}}Z_{i}^{\prime}-\frac{1}{16}N\right\vert \geq\frac{1}{32}\gamma N\right).
\]
Define the events 
\[
A=\left\{ \left\vert \frac{1}{4}\left\vert S_{c}\cap S_{d}\right\vert -\frac{N}{16}\right\vert \geq\frac{1}{160}\gamma N\right\} ,\;\;\;B=\left\{ \left\vert \sum\limits _{i\in S_{c}\cap S_{d}}Z_{i}^{\prime}-\frac{1}{4}\left\vert S_{c}\cap S_{d}\right\vert \right\vert \geq\frac{1}{40}\gamma N\right\} .
\]

Then, the probability can be further bounded by 
\[
\leq\mathbb{P}_{N}\left(A\right)+\mathbb{P}_{N}\left(B\right)\leq\frac{1}{2}e^{4}Ne^{-2N\left(\frac{1}{40}\gamma\right)^{2}}+2e^{-\frac{1}{2}N\left(\frac{1}{40}\gamma\right)^{2}}\leq e^{4}Ne^{-\frac{N\gamma^{2}}{3200}}
\]
where the first bound comes from Lemma \ref{lemT1}, and the second
from the second bound in Lemma \ref{lemT2}.

The remaining bounds have proofs similar (and simpler) to the case
$Z=Y_{a,b}^{c,d}$. 
\end{proof}
\vspace{0.3cm}

Finally, we describe an event $E$ that collects these bounds. Recall
that $\alpha=1/25$, and define for each $a\neq b$ \ and $c\neq d$,
\begin{eqnarray*}
E_{a,b,c,d} & = & \left\{ \left\vert \frac{Y_{a,b}}{Y_{a}}-1\right\vert \leq2\alpha\right\} \cap\left\{ \left\vert \frac{Y_{a,b}^{c}}{Y_{a}^{c}}-1\right\vert \leq2\alpha\right\} \cap\left\{ \left\vert \frac{Y_{a}^{c,d}}{Y_{a}^{c}}-1\right\vert \leq2\alpha\right\} \cap\left\{ \left\vert \frac{Y_{a,b}^{c,d}}{Y_{a}^{c,d}}-1\right\vert \leq2\alpha\right\} \\
 &  & \left\{ \left\vert \frac{Y^{c,d}}{Y^{c}}-1\right\vert \leq2\alpha\right\} \cap\left\{ \left\vert \frac{Y_{a}^{c}}{Y^{c}}-1\right\vert \leq2\alpha\right\} \cap\left\{ \left\vert \frac{Y_{a}^{c,d}}{Y^{c,d}}-1\right\vert \leq2\alpha\right\} .
\end{eqnarray*}
Finally, let 
\[
E=\bigcap\limits _{a,b,c,d:a\neq b\ \text{and }c\neq d}E_{a,b,c,d}.
\]

\vspace{0.3cm}

\begin{lem}
\label{lembound} We have 
\[
\mathbb{P}_{N}(E)>1-7e^{4}N^{5}e^{-\frac{N}{2163200}}\xrightarrow[n\to\infty]{}1.
\]
\end{lem}

\begin{proof}
Take $\gamma=\frac{\alpha}{1+\alpha}=\frac{1}{26}$ and let 
\[
F_{a,b,c,d}=\bigcap\limits _{Z\in\{Y_{a},Y_{a,b},Y^{c,d},Y^{c,d},Y_{a}^{c},Y_{a,b}^{c},Y_{a}^{c,d},Y_{a,b}^{c,d}\}}\left\{ \left\vert Z-N\right\vert \leq\gamma N\right\} .
\]
It is easy to see that $F_{a,b,c,d}\subseteq E_{a,b,c,d}.$ The probability
that $F_{a,b,c,d}$ holds can be bounded from Lemma \ref{lemT3} (as
soon as $N\geq\frac{64}{\gamma}=1664$), as 
\[
\mathbb{P}_{N}\left(F_{a,b,c,d}\right)\geq1-7e^{4}Ne^{-\frac{N}{32.(260)^{2}}}.
\]
The result follows since there are less than $N^{4}$ ways of choosing
$(a,b,c,d)$. 
\end{proof}
\vspace{0.5cm}
 Computations using the bound of lemma \ref{lembound} show that $N=52.10^{6}$
is enough to have the existence of an appropriate Markov chain. So
one can take $\varepsilon=3.10^{-17}$ in the statement of theorem
\ref{thm3}. We conclude the proof of proposition \ref{pro3} by showing
that event $E$ implies conditions $UI1$ and $UI2.$ 
\begin{lem}
\label{lemT4} If event $E$ holds, then the conditions $UI1,UI2$
are satisfied. 
\end{lem}

\begin{proof}
We fix the law $\nu$ of the Markov chain on $A$ and assume that
it has been induced, as explained at the beginning of section \ref{existence},
by a transition matrix $P$ satisfying $E$. For $l\geq1$, we forget
about the state in $K$ and still denote by $u^{l}$ the marginal
of $u^{l}$ over $C^{l}\times D^{l}$. If $c=(c_{1},...,c_{l})\in C^{l}$
and $d=(d_{1},...,d_{l})\in D^{l}$, we have $u^{l}(c,d)=\nu(c_{1},d_{1},...,c_{l},d_{l})$.

Let us begin with condition UI2 which we recall here: for all $1\leq p\leq l$,
for all ${d}\in D^{l}$, for all $d^{\prime}\in D^{p-1}$, for all
$m\in\{1,...,p-1\}$ such that ${d}_{m}\neq d_{m}^{\prime}$, for
$r=2m-1,2m$, 
\[
u^{l}\left(\tilde{c}\smile_{r+1}d^{\prime}|\tilde{d}={d},\tilde{c}\smile_{r}d^{\prime}\right)\in[1/2-\alpha,1/2+\alpha],\;\;\hspace{3cm}(\ref{eq63})
\]
where $(\tilde{c},\tilde{d})$ is a random variable selected according
to $u^{l}$. The quantity $u^{l}\left(\tilde{c}\smile_{r+1}d^{\prime}|\tilde{d}={d},\tilde{c}\smile_{r}d^{\prime}\right)$
is thus the conditional probability of the event $(\tilde{c}$ and
$d'$ are nice at level $r+1$) given that they are nice at level
$r$ and that the signal received by player 2 is $d$. We divide the
problem into different cases.

\underline{Case $m>1$ and $r=2m-1$. }

Note that the events $\{\tilde{c}\smile_{2m}d'\}$ and $\{\tilde{c}\smile_{2m-1}d'\}$
can be decomposed as follows : 
\begin{eqnarray*}
\{\tilde{c}\smile_{2m-1}d'\} & = & \{\tilde{c}\smile_{2m-2}d'\}\cap\{X_{d'_{m-1},\tilde{c}_{m}}=1\},\\
\{\tilde{c}\smile_{2m}d'\} & = & \{\tilde{c}\smile_{2m-2}d'\}\cap\{X_{d_{m-1}^{\prime},\tilde{c}_{m}}=1\}\cap\{X_{\tilde{c}_{m},d'_{m}}=1\}.
\end{eqnarray*}
So $u^{l}\left(\tilde{c}\smile_{2m}d^{\prime}|\tilde{d}={d},\tilde{c}\smile_{2m-1}d^{\prime}\right)=u^{l}\left(X_{\tilde{c}_{m},d'_{m}}=1|\tilde{d}={d},\tilde{c}\smile_{2m-1}d^{\prime}\right)$,
and the Markov property gives: 
\begin{eqnarray*}
u^{l}\left(\tilde{c}\smile_{2m}d^{\prime}|\tilde{d}={d},\tilde{c}\smile_{2m-1}d^{\prime}\right) & = & u^{l}\left(X_{\tilde{c}_{m},d'_{m}}=1|X_{d'_{m-1},\tilde{c}_{m}}=1,X_{d_{m-1},\tilde{c}_{m}}=1,X_{\tilde{c}_{m},d_{m}}=1\right),\\
 & = & \frac{\sum_{i\in U}X_{i,d_{m}^{\prime}}X_{d'_{m-1},i}X_{d_{m-1},i}X_{i,d_{m}}}{\sum_{i\in U}X_{d'_{m-1},i}X_{d_{m-1},i}X_{i,d_{m}}}.
\end{eqnarray*}
This is equal to $\frac{1}{2}\frac{Y_{d_{m},d_{m}^{\prime}}^{d_{m-1},d_{m-1}^{\prime}}}{Y_{d_{m}}^{d_{m-1},d_{m-1}^{\prime}}}$
if $d'_{m-1}\neq d_{m-1}$, and to $\frac{1}{2}\frac{Y_{d_{m},d_{m}^{\prime}}^{d_{m-1}}}{Y_{d_{m}}^{d_{m-1}}}$
if $d_{m-1}^{\prime}=d_{m-1}$. In both cases, $E$ implies (\ref{eq63}).

\underline{Case $r=2m$.}

We have $u^{l}\left(\tilde{c}\smile_{2m+1}d'|\tilde{d}=d,\tilde{c}\smile_{2m}d^{\prime}\right)=u^{l}\left(X_{d_{m}^{\prime},\tilde{c}_{m+1}}=1|\tilde{d}=d,\tilde{c}\smile_{2m}d^{\prime}\right)$,
and by the Markov property : 
\begin{eqnarray*}
u^{l}\left(\tilde{c}\smile_{2m+1}d'|\tilde{d}=d,\tilde{c}\smile_{2m}d^{\prime}\right) & = & u^{l}\left(X_{d_{m}^{\prime},\tilde{c}_{m+1}}=1|X_{{d_{m}},\tilde{c}_{m+1}}=1,X_{\tilde{c}_{m+1},d_{m+1}}=1\right),\\
 & = & \frac{\sum_{i\in U}X_{d_{m}^{\prime},i}X_{d_{m},i}X_{i,d_{m+1}}}{\sum_{i\in U}X_{d_{m},i}X_{i,d_{m+1}}}\\
 & = & \frac{1}{2}\frac{Y_{d_{m+1}}^{d'_{m},d_{m}}}{Y_{d_{m+1}}^{d_{m}}}\;\in[1/2-\alpha,1/2+\alpha].
\end{eqnarray*}

\underline{Case $m=1$, $r=1$.} 
\begin{eqnarray*}
u^{l}\left(\tilde{c}\smile_{2}d'|\tilde{d}=d,\tilde{c}\smile_{1}d^{\prime}\right) & = & u^{l}\left(\tilde{c}\smile_{2}d'|\tilde{d}=d\right),\\
 & = & u^{l}\left(X_{\tilde{c}_{1},d_{1}^{\prime}}=1|X_{\tilde{c}_{1},d_{1}}=1\right),\\
 & = & \frac{\sum_{i\in U}X_{i,d'_{1}}X_{i,d_{1}}}{\sum_{i\in U}X_{i,d_{1}}}\\
 & = & \frac{1}{2}\frac{Y_{d_{1},d_{1}^{\prime}}}{Y_{d_{1}}}\;\in[1/2-\alpha,1/2+\alpha].
\end{eqnarray*}

Let us now consider condition $UI1$: we require that for all $l\geq1$,
all ${c}=({c}_{1},...,{c}_{l})$ in $C^{l}$ and $c^{\prime}=(c_{1}^{\prime},...,c_{l+1}^{\prime})$
in $C^{l+1}$ such that ${c}_{1}=c_{1}^{\prime}$, we have 
\[
u^{l}\left(c^{\prime}\smile_{2l+1}\tilde{d}\;\big|\;\tilde{c}={c},c^{\prime}\smile_{2l}\tilde{d}\right)\in[1/2-\alpha,1/2+\alpha]\hspace{3cm}(\ref{eq61})
\]
and for all $m\in\{1,...,l\}$ such that ${c}_{m}\neq c_{m}^{\prime}$,
for $r=2m-2,2m-1$, 
\[
u^{l}\left(c^{\prime}\smile_{r+1}\tilde{d}\;\big|\;\tilde{c}=c,c^{\prime}\smile_{r}\tilde{d}\right)\in[1/2-\alpha,1/2+\alpha].\hspace{3cm}(\ref{eq62})
\]

We start with (\ref{eq61}). 
\begin{eqnarray*}
u^{l}\left(c^{\prime}\smile_{2l+1}\tilde{d}|\tilde{c}=c,c^{\prime}\smile_{2l}\tilde{d}\right) & = & u^{l}\left(X_{\tilde{d}_{l},c_{l+1}^{\prime}}=1|\tilde{c}=c,c^{\prime}\smile_{2l}\tilde{d}\right),\\
 & = & u^{l}\left(X_{\tilde{d}_{l},c_{l+1}^{\prime}}=1|X_{c_{l}^ {},\tilde{d}_{l}}=1,X_{c_{l}^{\prime},\tilde{d}_{l}}=1\right),\\
 & = & \frac{\sum_{i\in V}X_{i,c_{l+1}^{\prime}}X_{c_{l},i}X_{c_{l}^{\prime},i}}{\sum_{i\in V}X_{c_{l},i}X_{c_{l}^{\prime},i}}.
\end{eqnarray*}
This is $\frac{1}{2}\frac{Y_{c_{l+1}}^{c_{l},c_{l}^{\prime}}}{Y^{c_{l},c_{l}^{\prime}}}$
if $c_{l}^{\prime}\neq c_{l}$, and $\frac{1}{2}\frac{Y_{c_{l+1}}^{c_{l}}}{Y^{c_{l}}}$
if $c_{l}^{\prime}=c_{l}$. In both cases, (\ref{eq61}) holds.

We finally consider (\ref{eq62}) and distinguish several case.

\underline{Case $r=2m-1$ and $m=l$.}

\begin{eqnarray*}
u^{l}\left(c^{\prime}\smile_{2l}\tilde{d}|\tilde{c}=c,c^{\prime}\smile_{2l-1}\tilde{d}\right) & = & u^{l}\left(X_{c_{l}^{\prime},\tilde{d}_{l}}=1|\tilde{c}=c,c^{\prime}\smile_{2l-1}\tilde{d}\right),\\
 & = & u^{l}\left(X_{c_{l}^{\prime},\tilde{d}_{l}}=1|X_{c_{l}^ {},\tilde{d}_{l}}=1\right),\\
 & = & \frac{\sum_{i\in V}X_{c'_{l},i}X_{c_{l},i}}{\sum_{i\in V}X_{c_{l},i}},\\
 & = & \frac{1}{2}\frac{Y^{c'_{l},c_{l}^ {}}}{Y^{c_{l}}}\;\in[1/2-\alpha,1/2+\alpha].
\end{eqnarray*}

\underline{Case $r=2m-1$ and $m<l$.}

\begin{eqnarray*}
u^{l}\left(c^{\prime}\smile_{2m}\tilde{d}|\tilde{c}=c,c^{\prime}\smile_{2m-1}\tilde{d}\right) & = & u^{l}\left(X_{c_{m}^{\prime},\tilde{d}_{m}}=1|\tilde{c}=c,c^{\prime}\smile_{2m-1}\tilde{d}\right),\\
 & = & u^{l}\left(X_{c_{m}^{\prime},\tilde{d}_{m}}=1|X_{c_{m},\tilde{d}_{m}}=1,X_{\tilde{d}_{m},c_{m+1}}=1\right),\\
 & = & \frac{\sum_{i\in V}X_{c'_{m},i}X_{c_{m},i}X_{i,c_{m+1}}}{\sum_{i\in V}X_{c_{m},i}X_{i,c_{m+1}}},\\
 & = & \frac{1}{2}\frac{Y_{c_{m+1}}^{c'_{m},c_{m}}}{Y_{c_{m+1}}^{c_{m}}}\;\in[1/2-\alpha,1/2+\alpha].
\end{eqnarray*}

\underline{Case $r=2m-2$.}

\begin{eqnarray*}
u^{l}\left(c^{\prime}\smile_{2m-1}\tilde{d}|\tilde{c}=c,c^{\prime}\smile_{2m-2}\tilde{d}\right) & = & u^{l}\left(X_{\tilde{d}_{m-1},c_{m}^{\prime}}=1|\tilde{c}=c,c^{\prime}\smile_{2m-1}\tilde{d}\right),\\
 & = & u^{l}\left(X_{\tilde{d}_{m-1},c_{m}^{\prime}}=1|X_{c'_{m-1},\tilde{d}_{m-1}}=X_{c_{m-1},\tilde{d}_{m-1}}=X_{\tilde{d}_{m-1},c_{m}}=1\right),\\
 & = & \frac{\sum_{i\in V}X_{i,c'_{m}}X_{i,c_{m}}X_{c'_{m-1},i}X_{c_{m-1},i}}{\sum_{i\in V}X_{i,c_{m}}X_{c'_{m-1},i}X_{c_{m-1},i}}.
\end{eqnarray*}
This is $\frac{1}{2}\frac{Y_{c'_{m},c_{m}}^{c'_{m-1},c_{m-1}}}{Y_{c_{m}}^{c'_{m-1}c_{m-1}}}$
if $c_{m-1}\neq c'_{m-1}$, and $\frac{1}{2}\frac{Y_{c'_{m},c_{m}}^{c_{m-1}}}{Y_{c_{m}}^{c_{m-1}}}$
if $c_{m-1}=c'_{m-1}$. In both cases, it belongs to $[1/2-\alpha,1/2+\alpha]$,
concluding the proofs of lemma \ref{lemT4}, proposition \ref{pro3}
and theorem \ref{thm3}. 
\end{proof}

\section{Proofs of Section \ref{sec:Relation-to-other}}

\subsection{Proof of Theorem \ref{thm:NZS distance}}

Suppose that $u$ and $v$ are two simple, and non-redundant information
structures. Let $\tilde{u}$ and $\tilde{v}$ be the associated probability
distributions over belief hierarchies of player 1. It is easy to show
that if two non-redundant information structures induce the same distributions
over hierarchies of beliefs $\tilde{u}=\tilde{v}$, then they are
equivalent from any strategic point of view, and, in particular, they
induce the same set of ex ante BNE payoffs. Hence, we assume that
$\tilde{u}\neq\tilde{v}$.

Let $H_{u}=\text{supp}\tilde{u}$ and $H_{v}=\text{supp}\tilde{v}$.
Lemma III.2.7 in \citet{mertens_sorin_zamir_2015} implies that the
sets $H_{u}$ and $H_{v}$ are disjoint.

It is well known that there exists a non-zero sum payoff function
$g^{\left(0\right)}:K\times\left(I\times I_{0}\right)\times J\rightarrow\left[-1,1\right]^{2}$
such that $I_{0}=H_{u}\cup H_{v}$ and such that the set of rationalizable
actions for player 1 of type $c\in C$ with hierarchy $h\left(c\right)$
is contained in the set $I\times\left\{ h\left(c\right)\right\} $.
In particular, in a Bayesian Nash equilibrium, each type of player
1 will report its hierarchy. Construct game $g^{\left(1\right)}:K\times\left(I\times I_{0}\right)\times\left(J\times\left\{ u,v\right\} \right)\rightarrow\left[-1,1\right]^{2}$
with payoffs 
\begin{align*}
g_{1}^{\left(1\right)}\left(k,i,i_{0},j,j_{0}\right) & =g_{1}^{\left(0\right)}\left(k,i,i_{0},j\right),\\
g_{2}^{\left(1\right)}\left(k,i,i_{0},j,j_{0}\right) & =\frac{1}{2}g_{2}^{\left(0\right)}\left(k,i,i_{0},j\right)+\begin{cases}
\frac{1}{2}, & \text{if }j_{0}=u\text{ and }i_{0}\in H_{u}\\
-\frac{1}{2}, & \text{if }j_{0}=u\text{ and }i_{0}\notin H_{u},\\
0, & \text{if }j_{0}=v.
\end{cases}
\end{align*}
Then, the rationalizable actions of player 2 in game $g^{\left(1\right)}$
are contained in $J\times\left\{ u\right\} $ for any type in type
space $u$ and in $J\times\left\{ v\right\} $ for any type in type
space $v$.

Finally, for any $\varepsilon\in\left(0,1\right),$construct a game
$g^{\varepsilon}:K\times\left(I\times I_{0}\right)\times\left(J\times\left\{ u,v\right\} \right)\rightarrow\left[-1,1\right]^{2}$
with payoffs 
\begin{align*}
g_{1}^{\varepsilon}\left(k,i,i_{0},j,j_{0}\right) & =\varepsilon g_{1}^{\left(0\right)}\left(k,i,i_{0},j,j_{0}\right)+\left(1-\varepsilon\right)\begin{cases}
1, & \text{if }j_{0}=u,\\
-1, & \text{if }j_{0}=v,
\end{cases},\\
g_{2}^{\varepsilon} & \equiv g_{2}^{\left(1\right)}.
\end{align*}
Then, the Bayesian Nash equilibrium payoff of player belongs to $\left[1-\varepsilon,1\right]$
on the structure $u$ and $\left[-1,-1+\varepsilon\right]$ on the
structure $v$. It follows that the payoff distance between the two
type spaces is at least $2-2\varepsilon$, for arbitrary $\varepsilon>0$.

Next, suppose that $u$ and $v$ are two simple information structures
with the decomposition $u=\sum_{\alpha}p_{\alpha}u_{\alpha}$ and
$v=\sum_{\alpha}q_{\alpha}v_{\alpha}$ and such that $\tilde{u}_{\alpha}=\tilde{v}_{\alpha}$
for each $\alpha$. Let $g$ be a non-zero sum payoff function. Let
$\sigma_{\alpha}$ be an equilibrium on $u_{\alpha}$ with payoffs
$g_{\alpha}\in\R^{2}$. Let $s_{\alpha}$ is the associated equilibrium
on $v_{\alpha}$ (that can be obtained by mapping the hierarchies
of beliefs through an appropriate bijection) with the same payoffs
$g_{\alpha}$. The distance between payoffs is bounded my 
\[
\left|\sum p_{\alpha}g_{a}-q_{\alpha}g_{\alpha}\right|_{\max}=\left|\sum\left(p_{\alpha}-q_{\alpha}\right)g_{\alpha}\right|_{\max}\leq\sum\left|p_{\alpha}-q_{\alpha}\right|,
\]
where the last inequality comes from the fact that payoffs are bounded.

On the other hand, let $A=\left\{ \alpha:p_{\alpha}>q_{\alpha}\right\} $.
Using a similar construction as above, we can construct a game $g^{\left(1\right)}$
such that player 2's actions have a form $J\times\left\{ u_{A},u_{B}\right\} $,
and his rationalizable actions are contained in set $J\times\left\{ u_{A}\right\} $
for any type in type space $u_{\alpha},\alpha\in A$ and in $J\times\left\{ u_{B}\right\} $
otherwise. Further, we construct a game $g^{\left(\varepsilon\right)}$
as above. Then, any player $1$'s equilibrium $g_{1,\alpha}^{(\varepsilon)}$
payoff is at least $1-\varepsilon$ for any type in type space $u_{\alpha},\alpha\in A$,
and $-1+\varepsilon$ for any type in type space $u_{\alpha}$ for
$\alpha\notin A$. Denoting the equilibrium payoff of player 2 as
$g_{2,\varepsilon}^{\varepsilon}$, the payoff distance in game $g^{\varepsilon}$
is at least 
\begin{align*}
 & \max\left(\left|\sum_{\alpha}\left(p_{\alpha}-q_{\alpha}\right)g_{1,\alpha}\right|,\left|\sum_{\alpha}\left(p_{\alpha}-q_{\alpha}\right)g_{2,\alpha}\right|\right)\\
\geq & \left|\sum_{\alpha}\left(p_{\alpha}-q_{\alpha}\right)g_{1,\alpha}\right|\\
\geq & \left[\sum_{\alpha\in A}\left(p_{\alpha}-q_{\alpha}\right)-\sum_{\alpha\notin A}\left(p_{\alpha}-q_{\alpha}\right)\right]\left(1-\varepsilon\right)\\
\geq & \left(1-\varepsilon\right)\sum\left|p_{\alpha}-q_{\alpha}\right|.
\end{align*}
Because the $\varepsilon>0$ is arbitrary, the two above inequalities
show that the payoff distance is equal to $\sum\left|p_{\alpha}-q_{\alpha}\right|$.

 \bibliographystyle{apa}
\bibliography{zerosum2}

\end{document}